\documentclass[a4paper,notitlepage, 8pt,reqno]{article}

\usepackage[cp1251]{inputenc}
 \usepackage[T2A]{fontenc}
 \usepackage[russian,english]{babel}
\usepackage[all,arc,poly,2cell,curve,arrow,tips]{xy}
\usepackage{enumerate, eucal, amsthm,amsmath, amssymb}
\usepackage{graphicx}
\usepackage{mathtext}

\usepackage{tikz}

\allowdisplaybreaks[4]

\theoremstyle{definition}

\theoremstyle{plain}
\newtheorem{theorem}{Theorem}
\newtheorem*{thm*}{Theorem}
\newtheorem{prop}{Proposition}
\newtheorem*{prop*}{Proposition}
\newtheorem{corollary}{Corollary}
\newtheorem{lemma}{Lemma}

\newcommand{\const}{\operatorname{const}}

\begin{document}

	\title{Clebsh-Gordan coefficients for the algebra  $\mathfrak{gl}_3$ and hypergeometric functions \footnote{   Это препринт Произведения, принятого для публикации в журнале «Алгебра и Анализ», 2021 год. Владелец прав на распространение – ПОМИ РАН.  This is a preprint of a paper submitted to "St. Petersburg Mathematical Journal". All rights belong  to  POMI RAS  }}
	
	\author{D.~V.~Artamonov}

\maketitle

\begin{abstract}
	The Clebsh-Gordan coefficients for the Lie algebra
	$\mathfrak{gl}_3$ in the Gelfand-Tsetlin base are  calculated. In contrast to previous papers the result is given as an explicit formula.   To obtain the result a realization of a representation in the space of functions on the group  $GL_3$ is used. The keystone fact that allows  to carry the calculation of Clebsh-Gordan coefficients is  the theorem that says that functions  corresponding to Gelfand-Tsetlin base vectors can be expressed through  generalized hypergeometric functions.
\end{abstract}

\section{Introduction}

Let  $U$, $V$ --- be finite dimensional irreducible representation of the Lie algebra~$\mathfrak{gl}_N$. Let us take their tensor product and split it into a sum of irreducible representations
\begin{equation}\label{rcg} U\otimes V=\sum_{\alpha} W^{\alpha}.
\end{equation}

Let $\{u_i\}$, $\{v_j\}$  be bases in  $U$, $V$,  and let
$\{w_k^{\alpha}\}$ be a base in
$W^{\alpha}$.
One has an relation
\begin{equation}
\label{Clgo}
w_k^{\alpha}=\sum_{i,j}C^{i,j}_k(\alpha)u_i\otimes v_j,\quad  C^{i,j}_k(\alpha)\in\mathbb{C}.
\end{equation}

The coefficietns $C^{i,j}_k(\alpha)$  in this relation are called {\it  the Clebsh-Gordan coefficients}.

Below we discuss only the cases $N=2,3$.  In the representations we take a  Gelfand-Tsetlin base, since this type of base appeares naturally in applications that are discussed below.

The Clebsh-Gordan coefficients for $\mathfrak{gl}_2$  play an important role in the quantum mechanics  in the theory of spin. These coefficients were calculated explicitly by van der Waerden and Racah  (see \cite{Gkl0,Gkl}).

The Clebsh-Gordan coefficients  for the algebra  $\mathfrak{gl}_3$  play an important role in the theory of quarks (see  \cite{qq}).
But the problem of calculation of the Clebsh-Gordan coefficients in the case  of the algebra  $\mathfrak{gl}_3$  is a much more difficult problem than in the case
$\mathfrak{gl}_2$.   Nevertheless in some sence the  formulas were obtained.  Firstly it was done by Biedenharn, Baird, Louck   and others in the series of papers~\cite{bb1963,bl1968,bl1970,bl19731,bl19732}\footnote{Mostly in these papers the case  $\mathfrak{gl}_N$ is considered, but  only in the case $\mathfrak{gl}_3$ the obtained results allow to obtain in principle a  formula for an arbitrary  Clebsh-Gordan coefficient.}.
But the obtained results are  bulky. Also the papers do not contain a direct formula   of type  $C^{i,j}_k(\alpha)=\dots$\,. It is only clear that such a formula can in principle be obtained from the results of these papers.

At the same time there were appearing numerous papers  where these coefficients were calculated in some particular cases
(see \cite{mo,l1,he,kl1,bb65}).  There are attempts to make the formulas easier (in some particular cases) by using special functions   (see \cite{al,blsf}).

Since an explicit result was not obtained investigations were continuing. One should mention the paper
\cite{h}, where the authors announced the discovery of an explicit formula for the Clebsh-Gordan coefficients for the algebra
я $\mathfrak{gl}_3$, but in this paper an answer is bulky and actually it is not explicit. The paper does not contain a formulas of type
$C^{i,j}_k(\alpha)=\dots$.   So even later there were appearing papers (see \cite{h1}) devoted to the search of an explicit formula for the Clebsh-Gordan coefficients.

Since a simple or at least an explicit formula was not obtained there appeared papers devoted to algorithmic caculation of Clebsh-Gordan coefficients for   $\mathfrak{gl}_3$, see
\cite{ww,rore}.  Let us especially note the paper  \cite{mm},  where a code of a program that calculates an arbitrary Clebsh-Gordan coefficient for  $\mathfrak{gl}_N$ is given. Moreover this algorithm is realized as an online calculator  \cite{onlmm}.

A starting point for the present paper is the following. In  \cite{bb1963}
a very interesting formula was derived. If one realizes an irreducible representation of   $\mathfrak{gl}_3$ in the space of functions on  $GL_3$,
then a function corresponding to a Gelfand-Tsetlin base vector can be expressed using a Gauss' hypegeometric function  $F_{2,1}$.  Also in
\cite{bb1963} a derivation of formulas for the action of generators of the algebra is sketched, which is bases on using of contiguity relations for the function    $F_{2,1}$.

The aim of the present paper is to obtain explicit for the Clebsh-Gordan coefficients for  $\mathfrak{gl}_3$ using the formulas expressing a function corresponding to a base vector through hypergeometric functions.

In contrast to  the paper  \cite{bb1963} to express the functions corresponding to Gelfand-Tsetlin diagrams  we use not the functions  $F_{2,1}$ but the hypergeometric  $\Gamma$-series.   It's construction has an important advantage over the function
$F_{2,1}$: such a series remembers all it's parameters.   This means the following. The considered  $\Gamma$-series satisfies a system of PDE called the Gelfand-Kapranov-Zelevinsky system (GKZ shortly).
The parameters  of a  $\Gamma$-series give asymptotic of it's  behaviour near a components of a singular locus of the GKZ system .
This fact is used in the present paper to obtain relation for the     $\Gamma$-series that allow to obtain explicit formulas for the Clebsh-Gordan coefficients.

The structure of the paper is the following. In  \S\ref{opk}  the basic notions are introduced . In ~\ref{tppk}  we presnet an important result: an explicit construction of highest vectors of representations  $W^{\alpha}$  from
\eqref{rcg}.  Before a close problem of multiplicities was discussed
(см.~\cite{mul1,mul2}), i.e. a problem of construction of an index  $\alpha$  in   \eqref{rcg}.  However explicit construction of highest vectores was not known. These vectors are of two types.

The case when the representation  $W^{\alpha}$ has the highest vector of the first type  is considered in  \S\ref{sl1}. The result of the calculation of the  Clebsh-Gordan coefficients in presented in Theorems    \ref{mnt11} and  \ref{mnt12}.

The case when the representation   $W^{\alpha}$ has the highest vector of the first type  is considered in ~\ref{cg2}.
Note that this case is reduced to the previous one.  The result of the calculation  in presented in Theorems \ref{mnt21} and~\ref{mnt22}.

The resulting formulas are not simple but they are really explicit.

\section{The basic notions and construction}
\label{opk}

\subsection{$\Gamma$-series} Information about a $\Gamma$-series can be found in~\cite{GG}.

Let $B\subset \mathbb{Z}^N$  be a lattice, let $\gamma\in \mathbb{Z}^N$ be a fixed vector. Define  {\it  a hypergeometric
	$\Gamma$-series} in variables $z_1,\dots,z_N$ as follows

\begin{equation}
F_{\gamma,B}(z)=\sum_{b\in
	B}\frac{z^{b+\gamma}}{\Gamma(b+\gamma+1)},
\end{equation}
where  $z=(z_1,\dots,z_N)$. We use a multi-index notation:
$$
z^{b+\gamma}:=\prod_{i=1}^N
z_i^{b_i+\gamma_i},\quad \Gamma(b+\gamma+1):=\prod_{i=1}^N\Gamma(b_i+\gamma_i+1).
$$

We need the following properties of a  $\Gamma$-series.

1)\,
A vector~$\gamma$ can be changed to  $\gamma+b$, $b\in B$, under this transformation  the series does not change.

2)\, $\frac{\partial }{\partial
	z_i}F_{\gamma,B}(z)=F_{\gamma-e_i,B}(z)$, where  $e_i=(0,\dots,1_{\text{
		at the place }i},\dots,0)$.

\smallskip

Below  we put  $N=4$, $\gamma=(\gamma_1,\gamma_2,\gamma_3,0)$,  and
$B=\mathbb{Z}\langle(1,-1,-1,1)\rangle$.

3)\,
Let  $F_{2,1}(a_1,a_2,b_1;z)=\sum_{n\in\mathbb{Z}^{\geq
		0}}\frac{(a_1)_n(a_2)_n}{(b_1)_n}z^n$, where
$(a)_n=\frac{\Gamma(a+n)}{\Gamma(a)}$, is a  Gauss' hypergeometric series. Then
\begin{align*}
F_{\gamma,B}(z_1,z_2,z_3,z_4)&=cz_1^{\gamma_1}z_2^{\gamma_2}z_3^{\gamma_3}F_{2,1}\Big(-\gamma_2,-\gamma_3,\gamma_1+1;\frac{z_1z_4}{z_2z_3}\Big),
\\
c&=\frac{1}{\Gamma(\gamma_1+1)\Gamma(\gamma_2+1)\Gamma(\gamma_3+1)}.
\end{align*}

4)\,  One has
$F_{2,1}(a,b,c;1)=\frac{\Gamma(c)\Gamma(c-a-b)}{\Gamma(c-a)\Gamma(c-b)}$.
Put  ${\bf 1}=(1,1,1,1)$, then for  $\gamma_4=0$
\begin{equation}
\label{fot1} F_{\gamma,B}({\bf
	1})=\frac{1}{\Gamma(\gamma_1+1)\Gamma(\gamma_2+1)\Gamma(\gamma_3+1)}
\cdot\frac{\Gamma(\gamma_1+1)\Gamma(\gamma_1+\gamma_2+\gamma_3+1)}{\Gamma(\gamma_1+\gamma_2+1)\Gamma(\gamma_1+\gamma_3+1)}.
\end{equation}

Also note that
\begin{equation}
\label{fot2}
F_{\gamma,B}(z_1,z_2,z_3,z_4)\mid_{z_1z_4=z_2z_3}=z_1^{\gamma_1}z_2^{\gamma_2}z_3^{\gamma_3}F_{\gamma,B}({\bf
	1}).
\end{equation}

5)\, A $\Gamma$-series satisfies the Gelfand-Kapranov-Zelevinsky system of PDE
\begin{align}
\begin{split}
\label{gkzs} \Big(\frac{\partial^2}{\partial z_1\partial
	z_4}-\frac{\partial^2}{\partial z_2\partial z_3}\Big)F_{\gamma,B}&=0,
\\
z_1\frac{\partial}{\partial z_1}F_{\gamma,B}+ z_2\frac{\partial}{\partial
	z_2}F_{\gamma,B}& =(\gamma_1+\gamma_2)F_{\gamma,B},\quad
\\
z_1\frac{\partial}{\partial z_1}F_{\gamma,B}+
z_3\frac{\partial}{\partial z_3}F_{\gamma,B}&=(\gamma_1+\gamma_3)F_{\gamma,B},
\\
z_1\frac{\partial}{\partial z_1}F_{\gamma,B}-z_4\frac{\partial}{\partial
	z_4}F_{\gamma,B}&=(\gamma_1-\gamma_4)F_{\gamma,B}.
\end{split}
\end{align}
Note that a singular locus of this system is  defined by the equations
\begin{equation}
z_{1}z_{2}z_{3}z_{4}(z_{1}z_{4}-z_{2}z_{3})=0.
\end{equation}

A sum of a  $\Gamma$-series is called a  $A$-hypergoemetric function.  Note that the   $\Gamma$-series considered in the paper are actually finite sums.

\subsection{A realization of a representation }
In the paper a realization of a representation of the Lie algebra
$\mathfrak{gl}_3$ in the space of function on the Lie group    $GL_3$.
On a function   $f(g)$,  where  $g\in GL_3$,  an element  $X\in GL_{3}$ acts by left shifts
\begin{equation}
\label{xf} (Xf)(g)=f(gX).
\end{equation}

Passing to an  infinitesimal action we obtain an action  of   $\mathfrak{gl}_3$ on the space of functions.

Let  $a_i^j$  be a function of a matrix element\footnote{The notation   $a_{i}^j$  for the  matrix elements is chosen by analogy with the papers \cite{bb1963,bl1968,bl1970,bl19731,bl19732};  in these papers a  bosonic realization is used, but this realization is equivalent to a realization in the space of functions on  $GL_3$.}, occurring in the row
$j$ and the column $i$. Introduce determinants
\begin{equation}
a_{i_1,\dots,i_k}:=\det(a_i^j)_{i=i_1,\dots,i_k}^{j=1,\dots,k}.
\end{equation}
In other words this is a determinant of a submatrix  in the matrix  $(a_i^j)$,
formed by rows   $1,\dots,k$  and columns
$i_1,\dots,i_k$.  As ~$a_i^j$, this a function on  $GL_3$.

An operator $E_{i,j}$ acts onto a determinant by transforming the column indices
\begin{equation}
\label{edet1}
E_{i,j}a_{i_1,\dots,i_k}=a_{\{i_1,\dots,i_k\}\mid_{j\mapsto i}},
\end{equation}where  $.\mid_{j\mapsto i}$ a substitution of an index  $j$ by
$i$, if the index~$j$ does not occur in   $\{i_1,\dots,i_k\}$,  then we put
$.\mid_{j\mapsto i}$ equal to zero. One has.
\begin{prop}
	The function \begin{equation} \label{stva}
	\frac{a_{1}^{m_{1}-m_{2}}}{(m_1-m_2)!}\frac{a_{1,2}^{m_2}}{m_2!}\end{equation}
	is a highest vector with the weight  $[m_1,m_2,0]$.
\end{prop}

We divide in  \eqref{stva} 
by  $(m_1-m_2)!m_2!$
to obtain simpler formulas below.

Also we use a realization in the space of function on a subgroup  $Z$ of matrices of type
\begin{equation}
\label{zcoor}
\begin{pmatrix}  1& x & y \\ 0&1 & z \\ 0 & 0 &1
\end{pmatrix},\end{equation}  i.e. in the space of functions of type $f(x,y,z)$; see  \cite{zh}.

\subsection{Tensor products}
\label{tppk}

A tensor product of representations can be realized in the space of functions on a product   $GL_3\times GL_3$. Let   $a_{i}^j$  be a matrix element of the first factor
$GL_3$,  and  let  $b_{i}^j$    be a matrix element of the second factor
$GL_3$.

In the previous Section we introduced determinants  $a_{i_1,\dots,i_k}$,
analogously one can define determinants  $b_{i_1,\dots,i_k}$. Let us also introduce the following expressions
\begin{align*}
(ab)_{i_1,i_1}&=\det\begin{pmatrix} a^1_i\\b^1_i
\end{pmatrix}_{i=i_1,i_2},\quad
(aabb)_{i_1,i_2,i_1,i_3}=a_{i_1,i_2}b_{i_1,i_3}-a_{i_1,i_3}b_{i_1,i_2},
\\
(aab)&=\det\begin{pmatrix} a^1_i\\a^2_i\\b^1_i
\end{pmatrix}_{i=1,2,3},
\,\,\qquad\qquad(abb)=\det
\begin{pmatrix}
a^1_i\\b^1_i\\b^2_i
\end{pmatrix}_{i=1,2,3}.
\end{align*}
Generators $E_{i,j}$ act onto
these determinants by changing the column indices by the formulas analogous to~\eqref{edet1}.

Take a tensor product of representations of $gl_3$ with the highest weights  $[m_1,m_2,0]$ and  $[\bar{m}_1,\bar{m}_2,0]$  with  the highest vectors of type \eqref{stva}.  Split this tensor product into a sum of  $\mathfrak{gl}_3$-irreducibles.

\begin{theorem}
	In the space of  $\mathfrak{gl}_3$-highest vectors there exists a base consisting of vectors
	\begin{equation}
	\label{stv}
	a_1^{\alpha}b_1^{\beta}a_{1,2}^{\gamma}b_{1,2}^{\delta}(ab)_{1,2}^{\omega}(abb)^{\varphi}(aabb)^{\theta}_{1,2,1,3};
	\end{equation}
	where
	\begin{align}
	\begin{split}
	\label{aop}
	\alpha+\omega+\varphi&=m_1-m_2,\quad \gamma+\theta=m_2,
	\\
	\beta+\omega&=\bar{m}_1-\bar{m}_2,\quad
	\delta+\varphi+\theta=\bar{m}_2.
	\end{split}
	\end{align}
\end{theorem}

\proof
{\bf 1)} The vectors  \eqref{stv}  belong to the tensor product of representations with the highest vectors  $ a_{1}^{m_1-m_2}a_{1,2}^{m_2}$
and  $ b_{1}^{\bar{m}_1-\bar{m}_2}b_{1,2}^{\bar{m}_2}$.
Indeed a space of a representation with the highest  vector  $
a_{1}^{m_1-m_2}a_{1,2}^{m_2}$  is a space of polynomials in $a_1,a_2,a_3, a_{1,2},a_{1,3},a_{2,3}$,  such that for each monomial one has   $\deg_{a_1}\!\!+\!\deg_{a_2}\!\!+\!\deg_{a_3}\!=\!m_1\!-\!m_2$,
${\deg_{a_{1,2}}\!\!+\!\deg_{a_{1,3}}\!\!+\!\deg_{a_{2,3}}\!=\!m_2}$ (see~\cite{zh}).
Analogously  one can describe a representation with the highest vector   $
b_{1}^{\bar{m}_1-\bar{m}_2}b_{1,2}^{\bar{m}_2}$.  Due to relations
\eqref{aop},  the function  \eqref{stv}  is a linear combination of products of  a polynomial in  $a$ and a polynomial in  $b$, with satisfy the written conditions.

\noindent{\bf 2)}  From a formula of type  \eqref{edet1} it immediately follows that
\eqref{stv} is a highest vector for  $\mathfrak{gl}_3$.

\noindent{\bf 3)} The vectors  \eqref{stv}  are linearly independent. Indeed, let us pass to a realization on the group  $Z\times Z$. Introduce the coordinate $x_1,y_1,z_1$ и  $x_2,y_2,z_2$  on the factors by analogy with  \eqref{zcoor}, one has:
\begin{align}
\begin{split}
\label{abbz}
a_1,b_1,a_{1,2},b_{1,2}&\mapsto 1,
\\
(abb)& \mapsto (x_2-x_1)z_2+x_{1}-y_{1},
\\
(ab)_{1,2}&\mapsto (x_{2}-x_{1}),
\\
(aabb)_{1,2,1,3} &\mapsto (z_{2}-z_{1}).
\end{split}
\end{align}
Hence functions  $(ab),(abb),(aabb)_{1,2,1,3}$ are algebraically independent, that is why the functions  \eqref{stv}  are linear independent.

\noindent{\bf 4)} Let us prove that linear combinations of  \eqref{stv} give all
$\mathfrak{gl}_3$-highest vectors.  To do it let us prove the equality of dimensions of the space of  $\mathfrak{gl}_3$-highest vectors and  of the span
vectors  \eqref{stv}.

\noindent{\bf 4.1)}  Let us give an explicit description of the space of  $\mathfrak{gl}_3$-highest vectors 
using a technique from \cite{zh}.  Elements of the tensor product we realize in the space of functions  $f$ on
$Z\times Z$.  This function belongs to a representation if and only if    $f$  satisfies the indicator sysytems in variables  $x_1,y_1,z_1$ and  $x_2,y_2,z_2$:
\begin{align}
\begin{split}
\label{is} & \begin{cases}L_1^{m_1-m_2+1}f=0,\\
L_2^{m_2+1}f=0,\end{cases}\text{ where }\quad
L_1=\frac{\partial}{\partial
	x_1}+z_1\frac{\partial}{\partial
	y_{1}},\quad L_2=\frac{\partial}{\partial z_{1}},\quad
\\
& \begin{cases}
\mathcal{L}_1^{\bar{m}_1-\bar{m}_2+1}f=0,\\\mathcal{L}_2^{\bar{m}_2+1}f=0,
\end{cases}\text{ where }\quad  \mathcal{L}_1=\frac{\partial}{\partial
	x_2}+z_2\frac{\partial}{\partial
	y_{2}},\quad \mathcal{L}_2=\frac{\partial}{\partial z_{2}}.
\end{split}
\end{align}

The function  $f$ is a
$\mathfrak{gl}_3$-highest vector if and only if it is invariant under the right action of the group  $diag(Z\times
Z)$.  Such a function can be written as  $f(\zeta,z)=g(\zeta h^{-1})$, $z\times \zeta\in Z\times Z$. Note that
\begin{align}
\label{zetaz} h=z\zeta^{-1}=\begin{pmatrix}  1 & x_{1}-x_{2} &
(y_{1}-y_{2})-z_{2}(x_{1}-x_{2})\\ 0 &1 & y_{1}-y_{2}
\\ 0& 0& 1
\end{pmatrix}.
\end{align}

The conditions
\eqref{is}  for the function  $f(\zeta,z)$  are equivalent to the following conditions \footnote{In
	\cite{zh}  the derivation of this result contains a mistake.} for the function  $g(h)$ of a matrix
$
h={\small\begin{pmatrix}1 & x_3 & y_3 \\0 &1& z_3\\ 0& 0 &1 \end{pmatrix}:}
$
\begin{align}
\begin{split}
\label{is1} &
\begin{cases}L_1^{m_2-m_1+1}g=0,\\L_2^{m_1+1}g=0,\end{cases}
\text{ where  }L_1=\frac{\partial}{\partial
	x_3}+z_{3}\frac{\partial}{\partial y_{3}}
,\quad L_2=\frac{\partial}{\partial
	z_{3}};
\\
&
\begin{cases}R_1^{\bar{m}_1+1}g=0,\\R_2^{\bar{m}_2-\bar{m}_1+1}g=0,\end{cases}
\text{ where }R_1=\frac{\partial}{\partial
	z_{3}}+x_{3}\frac{\partial}{\partial y_{3}},\quad
R_2=\frac{\partial}{\partial x_{3}}.
\end{split}
\end{align}
To find a base in the space of solution of  \eqref{is1} introduce variables
\begin{align}
\begin{split}
\label{uvw1} &u=y_{3}-x_{3}z_{3},\quad
v=y_{3}+x_{3}z_{3},\quad w_1=z_{3},\quad
w_2=x_{3}.
\end{split}
\end{align}
Note that  $(u+v)=2y_{3}$, $(v-u)=2x_{3}z_{3}$.
Instead of variables  $x_{3},y_{3},z_{3}$ one can use variables $u,v,w_1$ or  $u,v,w_2$. These two collections are related by the equality $w_2=\frac{v-u}{2w_1}$.

Introduce function
\begin{equation}
\label{w12}
w_1^A(u+v)^Bw_2^C=w_1^{A-C}(u+v)^{B}(v-u)^C=w_2^{C-A}(u+v)^B(v-u)^A.
\end{equation}
The space of polynomial solution of the system
$$
L_2^{m_2+1}g=L_1^{m_1-m_2+1}g=0
$$
has a base consisting of functions  \eqref{w12}, such that
\begin{equation}
\label{nerm}
A,B,C\geq 0,\quad A\leq m_2,\quad B+C\leq m_1-m_2.
\end{equation}
Analogously the space of polynomial solutions of the system
$$
R_2^{\bar{m}_1-\bar{m}_2+1}g=R_1^{\bar{m}_2+1}g=0
$$
has a base consisting of functions   \eqref{w12}, such that
\begin{equation}
\label{nerbarm}
A,B,C\geq 0,\quad C\leq \bar{m}_1-\bar{m}_2,\quad B+A\leq
\bar{m}_2.
\end{equation}

\noindent{\bf 4.2)}  Note that the vectors  \eqref{stv} are defined by nonnegative integers  $\omega,\varphi, \theta$, such that
\begin{equation}
\label{aopner}
\begin{split}
\omega+\varphi&\leq m_1-m_2,\quad \,\omega\leq
\bar{m}_1-\bar{m}_2,
\\
\theta &\leq m_2,\quad\quad \varphi +\theta \leq
\bar{m}_2.
\end{split}
\end{equation}

\noindent{\bf 4.3)}  To the inequalities  \eqref{aopner}  their correspond inequalities \eqref{nerm},
\eqref{nerbarm}.  The correspondence
$$
A\leftrightarrow \theta,\quad B\leftrightarrow
\omega,\quad C\leftrightarrow \varphi
$$
is a bijection between the solution spaces  of  \eqref{aopner}  and  \eqref{nerm}, \eqref{nerbarm}.  Hence the dimension of the span of vectors   \eqref{stv} equals to the dimension of the space of all $\mathfrak{gl}_3$-highest vectors.
\endproof

The formula  \eqref{stv} is non-symmetric: it involves  $\!(abb)$, but it does not involve~$\!(aab)$.   To obtain a symmetric formula one can operate as follows. Using a relations
$$
a_{1}a_{2,3}-a_2a_{1,3}+a_3a_{1,2}=0,\quad b_{1}b_{2,3}-b_2b_{1,3}+b_3b_{1,2}=0,
$$
one can obtain a relation
\begin{equation}
(ab)_{1,2}(aabb)_{1,2,1,3}=(aab)a_1b_{1,2}+(abb)a_{1,2}b_1.
\end{equation}
From these relations one can make the following conclusion. If one introduces a notation
\begin{align*}
f(\omega,\varphi,\psi,\theta)&=a_1^{\alpha}b_1^{\beta}a_{1,2}^{\gamma}
b_{1,2}^{\delta}(ab)_{1,2}^{\omega}(abb)^{\varphi}(aab)^{\psi}(aabb)^{\theta}_{1,2,1,3},
\\
\alpha+\omega+\varphi&=m_1-m_2,\quad  \gamma+\theta+\psi=m_2,
\\
\beta+\omega+\psi&=\bar{m}_1-\bar{m}_2,\quad
\delta+\varphi+\theta=\bar{m}_2,
\end{align*}
then one has a relation
\begin{align*}
&f(\omega,\varphi,\psi,\theta)=f(\omega-1,\varphi+1,\psi,\theta-1)+f(\omega-1,\varphi,\psi+1,\theta-1).
\end{align*}

Applying this transformation one  concludes that every highest vector can be expressed through  the vectors of type
$f(0,\varphi,\psi,\theta)$ and  $f(\omega,\varphi,\psi,0)$.  The fact that these vector are linear independent can be obtained  by restriction of the corresponding functions onto the subgroup of upper-triangular matrices.

\begin{prop}
	|in the space of $\mathfrak{gl}_3$-highest vectors  one has a base  \textup,
	consisting of vectors of type \textup:
	\begin{enumerate}
		\item
		\begin{equation}
		f(0,\varphi,\psi,\theta),\quad \begin{cases}\varphi,\psi,\theta\geq
		0,\\ \alpha+\varphi=m_1-m_2\quad  \gamma+\theta+\psi=m_2,\\
		\beta+\psi=\bar{m}_1-\bar{m}_2,\quad
		\delta+\varphi+\theta=\bar{m}_2,
		\end{cases}
		\end{equation}
		
		\item
		\begin{equation}
		f(\omega,\varphi,\psi,0),\quad \begin{cases}\omega,\varphi,\psi\geq
		0,\\\alpha+\omega+\varphi=m_1-m_2,\quad  \gamma+\theta+\psi=m_2,\\
		\beta+\omega+\psi=\bar{m}_1-\bar{m}_2,\quad
		\delta+\varphi+\theta=\bar{m}_2.
		\end{cases}
		\end{equation}
		
	\end{enumerate}
	
\end{prop}

\subsection{The Gelfand-Tsetlin base}
Let us return to representations of  $\mathfrak{gl}_3$ realized in the space of function on $GL_3$  with the highest vector  \eqref{stva}.
Let us give a formula for the functions corresponding to the Gelfand-Tsetlin base vectors. To fix a  normalization we take in the space of
$\mathfrak{gl}_2$-vectors the following base
\begin{equation}
\label{stv3}
\frac{a_3^{m_1-k_1}}{(m_1-k_1)!}\frac{a_1^{k_1-m_2}}{(k_1-m_2)!}\frac{a_{1,3}^{m_2-k_2}}{(m_2-k_2)!}\frac{a_{1,2}^{k_2}}{k_2!}.
\end{equation}

Note that this function can be rewritten as follows \cite{zh}:
\begin{equation*}
\frac{(E_{3,2})^{m_2-k_2}}{(m_2-k_2)!}\frac{\nabla_{3,1}^{m_1-k_1}}{(m_1-k_1)!}f,\quad
\nabla_{3,1}=E_{3,1}+(E_{1,1}-E_{2,2}+1)^{-1}E_{3,2}E_{2,1},
\end{equation*}
where the highest vector  $f$ is given by the formula  \eqref{stva}.  Now let us find a vector corresponding to an arbitrary  Gelfand-Tsetlin diagram:
\begin{align}
\begin{split}
\label{dgc}
\begin{pmatrix}
m_1 && m_2 &&0\\ &k_1&& k_2\\&&s
\end{pmatrix}.
\end{split}
\end{align}

\begin{theorem}\label{vec3}
	Put  $\!B\!=\!\mathbb{Z}\langle(1,\!-\!1,\!-\!1,1)\rangle,$
	${\gamma\!=\!(s_1\!-\!m_2,k_1\!-\!s_1,m_2\!-\!k_2,0),}$  then to the diagram  \eqref{dgc}
	there corresponds
	\begin{equation}
	\label{forvec2}
	\frac{a_3^{m_1-k_1}}{(m_1-k_1)!}\frac{a_{1,2}^{k_2}}{k_2!}F_{\gamma,B}(a_1,a_{2},a_{1,3},a_{2,3}).
	\end{equation}
\end{theorem}

In   \cite{bb1963}  a close formula is given but instead of a
$\Gamma$-series the function  $F_{2,1}$ is used.  Using formula relating
$F_{2,1}$ and $F_{\gamma}$, one immediately obtains the Theorem

\section{ The case of the  vector  $f(\omega,\varphi,\psi,0)$}
\label{sl1} In the space of
$\mathfrak{gl}_3$-highest vectors we consider the base vectors of type
\begin{small}
	\begin{equation}
	\label{stv1} f_0\!=\!\frac{ a_1^{\alpha}b_1^{\beta}a_{1,2}^{\gamma}
		b_{1,2}^{\delta}(ab)_{1,2}^{\omega}(abb)^{\varphi}(aab)^{\psi}}{\alpha!\beta!\gamma!\delta!\omega!},\quad
	\begin{cases}
	\alpha\!+\!\omega\!+\!\psi\!=\!m_1\!-\!m_2,\quad  \gamma\!+\!\varphi\!=\!m_2,\\
	\beta\!+\!\omega\!+\!\varphi\!=\!\bar{m}_1\!-\!\bar{m}_2,\quad
	\delta\!+\!\psi\!=\!\bar{m}_2.
	\end{cases}
	\end{equation}
\end{small}
\hspace{-2mm}
The weight of this vector  equals to
$$
[M_1,M_2,M_3]=[\alpha+\beta+\gamma+\delta+\omega+\varphi+\psi,\varphi+\psi+\omega,\varphi+\psi].
$$

To write a formula for a Clebsh-Gordan coefficient we need a formula for an arbitrary Gelfand-Tsetlin base vector in a representation defined by  \eqref{stv1}. A vector corresponding to a diagram
\begin{equation}
\label{Diag0}
\begin{pmatrix}
M_1&&M_2&&M_3\\
&M_1-T_1 && M_2-T_2\\
&&\quad M_1-T_1-S
\end{pmatrix},
\end{equation}
can be written as follows:
\begin{equation*}
\frac{E_{2,1}^{S}}{S!} \frac{E_{3,2}^{T_2}}{T_2!}
\frac{\nabla_{3,1}^{T_1}}{T_1!}f_0.
\end{equation*}
Let us find a function on   $GL_3\times GL_3$ corresponding to this diagram.  Note that all generators  $\mathfrak{gl}_3$,  that are not Cartan element, act onto   $(aab)$ and   $(abb)$  as zero.

The main difficulty is to write a formula for the action of    $\nabla_{3,1}^{T_1}$ onto  $f_0$. Firstly in Section  \ref{nab31}  we write a  formula for the action of  $\nabla_{3,1}$.  Then to obtain a formula for the action of    $\nabla_{3,1}^{T_1}$,   we derive some new relations for   $\Gamma$-series.
Using them in Section  \ref{vct} we write a formula for  function corresponding to   \eqref{Diag0}.

Below we consider functions of type
$F_{\gamma,B}(a_1,a_2,a_{1,3},a_{2,3})$ for different $\gamma$.  That is why we use a shorter notation
\begin{equation}
F_{\gamma}\equiv F_{\gamma,B}(a_1,a_2,a_{1,3},a_{2,3}).
\end{equation}

\subsection{  The action of  $\nabla_{3,1}$}
\label{nab31}

Let us give a formula for the action of operators
$\widetilde{\nabla}_{3,1}=((E_{1,1}-E_{2,2}+1)E_{3,1}+E_{3,2}E_{2,1})$ onto a function  $g=a_1^{\alpha}b_1^{\beta}a_{1,2}^{\gamma}b_{1,2}^{\delta}(ab)_{1,2}^{\omega}$.

\begin{lemma}
	\begin{align*}
	\widetilde{\nabla}_{3,1}g&=(\alpha+
	\beta+\gamma+\delta+\omega+1)\Big(a_3\frac{\partial}{\partial a_1}
	+b_{3}\frac{\partial}{\partial
		b_1}\Big)g
	\\
	&-(aab)\frac{\partial^2}{\partial a_{1,2}\partial
		b_1}g -(abb)\frac{\partial^2}{\partial a_{1}\partial
		b_{1,2}}g.
	\end{align*}
\end{lemma}

\proof
The operator   $E_{3,1}$ can be written as follows:
$$
a_3\frac{\partial}{\partial a_1}+b_3\frac{\partial}{\partial b_1}+a_{3,2}\frac{\partial}{\partial a_{1,2}}+b_{3,2}\frac{\partial}{\partial b_{1,2}}+(ab)_{3,2}\frac{\partial}{\partial (ab)_{1,2}}.
$$
Onto  $E_{3,1}g$ the operator   $(E_{1,1}-E_{2,2}+1)$ acts as a multiplication onto  $(\alpha+\beta)$.

The operators   $E_{3,2}E_{2,1}$ are written as follows:
\begin{align*}
a_3\frac{\partial}{\partial a_1}&+b_3\frac{\partial}{\partial b_1}+a_{2}a_{1,3}\frac{\partial^2}{\partial a_1\partial a_{1,2}}
\\
&+a_{2}b_{1,3}\frac{\partial^2}{\partial a_1\partial b_{1,2}} +b_{2}a_{1,3}\frac{\partial^2}{\partial b_1\partial a_{1,2}}
+b_{2}b_{1,3}\frac{\partial^2}{\partial b_1\partial b_{1,2}} \\
&+a_{2}(ab)_{1,3}\frac{\partial^2}{\partial a_1\partial (ab)_{1,2}} +b_{2}(ab)_{1,3}\frac{\partial^2}{\partial b_1\partial (ab)_{1,2}}.
\end{align*}
Summing these two operatora one obtains after a simplification:
\begin{equation*}
\begin{split}
\nabla_{3,1}g
&=
(\alpha+\beta+\gamma+\delta+1)
\\
&\times\Big(\alpha a_3a_1^{\alpha-1}b_1^{\beta}a_{1,2}^{\gamma}b_{1,2}^{\delta}(ab)_{1,2}^{\omega}
+\beta a_3a_1^{\alpha}b_1^{\beta-1}a_{1,2}^{\gamma}b_{1,2}^{\delta}(ab)_{1,2}^{\omega}\Big)
\\
&\quad-\beta\gamma
(aab)a_1^{\alpha}b_1^{\beta-1}a_{1,2}^{\gamma-1}b_{1,2}^{\delta}(ab)_{1,2}^{\omega}-\alpha\delta
(abb)a_1^{\alpha-1}b_1^{\beta}a_{1,2}^{\gamma}b_{1,2}^{\delta-1}(ab)_{1,2}^{\omega}.\qedhere
\end{split}
\end{equation*}
\endproof

Put
\begin{align*}
&O_1=a_3\frac{\partial}{\partial a_1}+b_{3}\frac{\partial}{\partial
	b_1},\quad O_2=(aab)_{}\frac{\partial^2}{\partial a_{1,2}\partial
	b_1}+(abb)_{}\frac{\partial^2}{\partial a_{1}\partial b_{1,2}}.
\end{align*}
Note that
\begin{equation}
(E_{1,1}-1)O_1=O_1E_{1,1},\quad (E_{1,1}-1)O_2=O_2E_{1,1},\quad O_1O_2=O_2O_1.
\end{equation}
Thus one has,
\begin{align*}
\widetilde{\nabla}_{3,1}^nf&=\sum_{k} c^{\alpha+\beta+\gamma+
	\delta+\omega}_k O_1^kO_2^{n-k}f,
\\
c^{h}_k&=(-1)^{n-k}\sum_{1\leq i_1<\dots<i_k\leq
	n}(h+1-i_1)\dots(h+1-i_k).
\end{align*}

Since $\nabla_{3,1}=(E_{1,1}-E_{2,2}+1)^{-1}\widetilde{\nabla}_{3,1}$, we obtain the following statement
\begin{corollary}
	\begin{align*}
	&\frac{\nabla_{3,1}^n}{n!}f=\sum_{k} d^{\alpha+\beta+\gamma+
		\delta+\omega}_{k,n-k} \frac{O_1^k}{k!}\frac{O_2^{n-k}}{(n-k)!}f,\\
	&d^{h}_{k,n-k}=(-1)^{n-k}k!(n-k)!\sum_{1\leq i_1<\dots<i_{n-k}\leq
		n}(h+1-i_1)^{-1}\dots\\&\dots(h+1-i_k)^{-1}.
	\end{align*}
\end{corollary}

\subsection{Relation for a  $\Gamma$-series}
Below we use the following relations for a   $\Gamma$-series.
\begin{lemma}
	One has a relation\textup:
	\begin{equation}
	\label{rel1}
	a_1^uF_{\gamma}=\sum_{\tau}
	Y_{\tau}(a_2a_{1,3}-a_1a_{2,3})^{p_{\tau}}F_{\gamma^{\tau}},
	\end{equation}
	
	\begin{align}
	\begin{split}
	\label{rel2}
	&\frac{(abb)^{{\lambda}}}{{\lambda}!}\frac{(aab)^{{\mu}}}{{\mu}!}\frac{(ab)_{1,2}^{\omega}}{\omega!}
	=\sum_{\tau}
	X_{\tau}a_3^{u_{\tau}}a_{1,2}^{v^{\tau}}(a_2a_{1,3}-a_1a_{2,3})^{p_{\tau}}
	\\
	&\qquad\qquad\qquad\qquad\quad\times b_{3}^{g_{\tau}}b_{1,2}^{h_{\tau}}(b_2b_{1,3}-b_1b_{2,3})^{q_{\tau}}F_{\theta^{\tau}}(a)
	G_{\vartheta^{\tau}}(b),
	\end{split}
	\end{align}
	
	\begin{equation}
	\label{rel3}
	\frac{E_{3,2}^n}{n!}\frac{a_3^{m_1-k_1}}{(m_1-k_1)!}\frac{a_{1,2}^{k_2}}{k_2!}F_{\gamma}
	=\sum_{\tau} Z_{\tau
	}a_3^{i_{\tau}}a_{1,2}^{j_{\tau}}(a_{1}a_{2,3}-a_{2}a_{1,3})^{r_{\tau}}F_{\varepsilon^{\tau}}.
	\end{equation}
	The index   $\tau$  runs through some set.
\end{lemma}

We need explicit formulas for the coefficients  $X$, $Y$, $Z$.  Introduce notations. Let  $\gamma=(\gamma_1,\gamma_2,\gamma_{1,3},\gamma_{2,3})$, put
\begin{equation}
\label{ptg}
\Pi_{p,\gamma}:= \prod_{t=p}^1\big(t(t+1)+t(\gamma_1+\gamma_2+\gamma_{1,3}+\gamma_{2,3})\big).
\end{equation}

\begin{prop}
	\label{py}
	In  \eqref{rel1}  for   $p_{\tau}=p\in \mathbb{Z}_{\geq 0}$ one has a unique summand. For it
	\begin{align*}
	\gamma^{\tau}_1&=\gamma_1+u,\quad \gamma^{\tau}_2=\gamma_2-p,\\
	\gamma^{\tau}_{1,3} &=\gamma_{1,3}-p,\quad  \gamma^{\tau}_{2,3}=0,\\
	Y_{{\tau}}&=Y^u_{\gamma,p}=\frac{\Gamma(u)}{\Gamma(u-k)}\cdot\frac{F_{\gamma}({\bf
			1})}{F_{\gamma^{\alpha}}({\bf 1})}\cdot
	\frac{1}{\Pi_{p,\gamma^{\tau}}}.
	\end{align*}
\end{prop}

\begin{prop}
	\label{px} In   \eqref{rel2} one has  $p_{\tau}=q_{\tau}=q\in \mathbb{Z}_{\geq 0}$. To obtain\textup, a formula for   $X_{\tau},$ fix a partition
	\begin{align}
	\begin{split}
	\label{pxf}
	q\!=:\!q_1\!\!+\!q_2&\!\!+\!q_3,\quad \varphi_1\!\!+\!\varphi_2\!\!+\!\varphi_3\!:=\!\lambda\!-\!q\!-\!q_2,\quad  \psi_1\!\!+\!\psi_2\!\!+\!\psi_3\!:=\!\mu\!-\!q_3\!-\!q,\\
	\text{ and put }
	\\&
\begin{cases} 
u=\varphi_3+q_2,\quad
	v=\psi_3+q_3,\quad \ \\
	g=\psi_3+q_3,\quad  h=\varphi_3+q_2,
\\
	\theta_1=\varphi_1+\omega_1-\psi_1,\,\,
	\theta_2=\varphi_2+\omega_2+\psi_1,\quad \\ \theta_{1,3}=\psi_2+\psi_1,\,\,
		\theta_{2,3}=0,
	\\
	\vartheta_1=\psi_1+\omega_2-\varphi_1,\quad \vartheta_2=\psi_2+\omega_1+\varphi_1,\quad 
	\\ \vartheta_{1,3}=\varphi_2+\varphi_1,\quad \vartheta_{2,3}=0,
	\\
	X_{\tau}=X^{\theta,\vartheta,q_i}_{\psi_i,\varphi_i,\omega_i}=\frac{(-1)^{\varphi_2+\psi_2+\omega_2+q_2+q_3}}{\Pi_{q,\theta}\Pi_{q,\vartheta}}\cdot  \frac{h^{\lambda,\mu,\omega}_{q_1,q_2,q_3}}{F_{\theta}
		({\bf 1})F_{\vartheta} ({\bf 1}) }
		\\
	\qquad\qquad\times
	\frac{q!}{\varphi_1!\varphi_2!\varphi_3!\psi_1!\psi_2!\psi_3!\omega_1!\omega_2!}
	,
	\\
	\text{where }  h^{\lambda,\mu,\omega}_{q_1,q_2,q_3}
	=\sum_{\{1,\dots,q\}=I_1\sqcup I_2\sqcup I_3, \,\,|I_j|=q_j}
	\\
	\qquad\qquad\qquad\times\prod_{j\notin I_2\sqcup I_3}((q-j)
	\\+(\lambda-\{\text{the number of $i_s\in I_2,$ such that  $i_s<j$}\})		\\+(\mu-\{\text{the number of$i_s\in I_3,$ such that $i_s<j$}\})+\omega).
	\end{cases}
	\end{split}
	\end{align}
\end{prop}

\begin{prop}
	\label{pz} In the formula  \eqref{rel3}  one has   $r\in \mathbb{Z}_{\geq 0}$ and
	\begin{align*}
	i_{\tau}&=m_1-k_1+n_1,\quad  j_{\alpha}=k_2-n_2,
	\\[2mm]
	\varepsilon_1^{\tau}&=\gamma_1,\quad \varepsilon_2^{\tau}=\gamma_2-n_1-r,
	\\
	\varepsilon_{1,3}^{\tau}&=\gamma_{1,3}+n_2-r,\quad  \varepsilon_{2,3}=0,
	\\[2mm]
	&Z_{\tau }=Z_{n_1,n_2,r}^{\gamma, m_1-k_1}=\frac{1}{\Pi_{r,\varepsilon^{\tau}}}\cdot
	\frac{F_{\gamma-n_2e_2}({\bf 1})}{F_{\gamma^{\tau}}({\bf 1})} \cdot
	\frac{1}{(m_1-k_1)!n_1!(k_2-n_2)!n_2!}.
	\end{align*}
\end{prop}

Th proofs of the results of this Section in given in Section   \ref{pr}.

\subsection{A vector of a tensor product corresponding to a diagram }\label{vct}
Apply an operator
\begin{equation}
\label{ponop}
\frac{E_{2,1}^{S}}{S!}\frac{E_{3,2}^{T_2}}{T_2!}\frac{\nabla_{3,1}^{T_1}}{T_1!}
\end{equation}
to the vector  \eqref{stv1}. Note that
$$
\alpha+\beta+\gamma+ \delta+\omega=M_1-M_3.
$$
The result of application of  \eqref{ponop} to  \eqref{stv1} can be calculated by several steps.

\noindent{\bf 1)} After application of
$\frac{\nabla_{3,1}^{T_1}}{T_1!}$  to  \eqref{ponop} one gets
\begin{equation}
\begin{split}
\label{bol1}
\sum_{k_1+k_3+\varphi'+\psi'=T_1}&d^{M_1-M_3}_{k_1+k_3,\varphi'+\psi'}\frac{\lambda!}{\varphi!\varphi'!}
\frac{\mu!}{\psi!\psi'!}
\\
&\!\times\! \frac{a_1^{\alpha\!-\!k_1\!-\!\varphi'}}{(\alpha\!\!-\!k_1\!\!-\!\varphi')!}\frac{b_1^{\beta\!-\!k_3\!-\!\psi'}}{(\beta\!\!-\!k_3\!\!-\!\psi')!}
\frac{a_{1,2}^{\gamma\!\!-\!\!\psi'}}{(\gamma\!-\!\psi')!}\frac{b_{1,2}^{\delta\!\!-\!\varphi'}}{(\delta-\varphi')!}
\\
&\times
\frac{(ab)_{1,2}^{\omega}}{\omega!}
\frac{a_3^{k_1}}{k_1!}\frac{b_3^{k_3}}{k_3!}\frac{(aab)^{\lambda}}{\lambda!}\frac{(abb)^{\mu}}{\mu!},
\\
&\qquad\qquad\lambda=\varphi+\varphi',\quad \mu=\psi+\psi'.
\end{split}
\end{equation}

The summation is taken over all decompositions of   $T_1$, written below  $\sum$.

In the sums written  below the summation also is taken over all  partitions and some integer parameters.  When we move from  a sum to the next one the number of partitions increases and it not possible to  write them all under the sign   $\sum$.  We describe them in the test following the sum.

\noindent{\bf 2)} Apply to the product
$\frac{(ab)_{1,2}^{\omega}}{\omega!}\frac{(aab)^{(\psi+\psi')}}{(\psi+\psi')!}\frac{(abb)^{(\varphi+\varphi')}}{(\varphi+\varphi')!}$, containing in  \eqref{bol1},
the second main equality, using the  Plucker relation
$$
a_{1}a_{2,3}-a_2a_{1,3}=-a_{3}a_{1,2},\quad b_{1}b_{2,3}-b_2b_{1,3}=-b_{3}b_{1,2},
$$
one obtains
\begin{align}
\begin{split}
\label{bol2}
\sum
&d^{M_1-M_3}_{k_1+k_3,\varphi'+\psi'}\frac{\lambda!}{\varphi!\varphi'!}
\frac{\mu!}{\psi!\psi'!}
\\
&\times \frac{X_{\psi_i,\varphi_i,\omega_i}^{\theta,\vartheta,q_i}}{(\alpha-k_1-\varphi')!(\beta-k_3-\psi')!(\gamma-\psi')!(\delta-\varphi')!k_1!k_3!}
\\
&\times
a_3^{u+q+k_1}a_{1,2}^{v+q+\gamma-\psi'}b_{3}^{g+q+k_3}
b_{1,2}^{h+q+\delta-\varphi'}
\\
&\times
a_1^{\alpha-k_1-\varphi'}F_{\theta}(a)\cdot
b_1^{\beta-k_3-\psi'} G_{\vartheta}(b),
\end{split}
\end{align}

The summation in \eqref{bol2} is taken over all partitions of   $T_1$,  that appeared in the first step and over parameters  $q$, $\psi_i$, $\varphi_i$, $\omega_i$ defined in
\eqref{pxf}
According to  \eqref{pxf}, the parameters of the  $\Gamma$-series are the following:
\begin{align*}
\begin{cases}
u=\varphi_3+q_2,\quad  v=\psi_3+q_3,\quad \
g=\psi_3+q_3,\quad  h=\varphi_3+q_2,\\
\theta_1=\varphi_1+\omega_1-\psi_1,\quad
\theta_2=\varphi_2+\omega_2+\psi_1,\quad \theta_{1,3}=\psi_2+\psi_1,\quad \theta_{2,3}=0,\\
\vartheta_1=\psi_1+\omega_2-\varphi_1,\quad \vartheta_2=\psi_2+\omega_1+\varphi_1,\quad
\vartheta_{1,3}=\varphi_2+\varphi_1,\quad
\vartheta_{2,3}=0.\\
\end{cases}
\end{align*}

\noindent{\bf 3)}  Apply the first main equality to the following products from ~\eqref{bol2}:
{\allowdisplaybreaks
	\begin{align*}
	&a_1^{\alpha-k_1-\varphi'}F_{\theta}(a)=\sum_{t}
	Y^{\theta,\alpha-k_1-\varphi'}_t(a_1a_{2,3}-a_2a_{1,3})^{t}F_{\gamma^{t}}(a),\\
	&b_1^{\beta-k_3-\psi'} G_{\vartheta}(b)=\sum_{s}
	Y^{\vartheta,\beta-k_3-\psi'}_s(b_1b_{2,3}-b_2b_{1,3})^{s}G_{\delta^{s}}(b).
	\end{align*}
}

One obtains
\begin{align}
\begin{split}
\label{dddd}
&\sum
d^{M_1-M_3}_{k_1+k_3,\varphi'+\psi'}\frac{\lambda!}{\varphi!\varphi'!}
\frac{\mu!}{\psi!\psi'!}
\\
&\times \frac{(-1)^{t+s}X_{\psi_i,\varphi_i,\omega_i}^{\theta,\vartheta,q_i}\cdot Y^{\theta,\alpha-k_1-\varphi'}_t \cdot Y^{\vartheta,\beta-k_3-\psi'}_s}{(\alpha-k_1-\varphi')!(\beta-k_3-\psi')!(\gamma-\psi')!(\delta-\varphi')!k_1!k_3!}
\\
&\times
(u+q+k_1+t_{})!
\\
&\times (v+q+\gamma-\psi'+t_{})!(g+q+k_3+s_{})!(h+q+\delta-\psi'+s_{})!
\\
&\times
\frac{a_3^{u+q+k_1+t_{}}}{(u+q+k_1+t_{})!}\frac{a_{1,2}^{v+q+\gamma-\psi'+t_{}}}
{(v+q+\gamma-\psi'+t_{})!}
\frac{b_{3}^{g+q+k_3+s_{}}}{(g+q+k_3+s_{})!}
\\
&\times
\frac{b_{1,2}^{h+q+\delta-\varphi'+s_{}}}{(h+q+\delta-\varphi'+s_{})!}\cdot
F_{\gamma^{t}}(a)\cdot G_{\delta^{s}}(b).
\end{split}
\end{align}

The summation is taken over partitions that appeared in the previous steps and also over integer  non-negative parameters   $t$, $s$.
The parameters of the  $\Gamma$-series involved in this formulas are the following:
\begin{align*}
\begin{cases}
\gamma^t_1=\varphi_1+\omega_1-\psi_1+\alpha-k_1-\varphi',\,\,
\gamma^t_2=\varphi_2+\omega_2+\psi_1-t,\quad \\ \gamma^t_{1,3}=\psi_2+\psi_1-t,\quad
\gamma^t_{2,3}=0,
\\[2mm]
\delta^s_1=\psi_1+\omega_2-\varphi_1+\beta-k_3-\psi',\quad
\delta^s_2=\psi_2+\omega_1+\varphi_1-s,\quad \\ \delta^s_{1,3}=\varphi_2+\varphi_1-s,\quad  \delta^s_{2,3}=0.
\end{cases}
\end{align*}

\noindent{\bf 4)} Apply to  \eqref{dddd}  the operator
$\frac{E_{3,2}^{T_2}}{T_2!}$,  which action onto the tensor product can be written as follows:
$$
\frac{E_{3,2}^{T_2}}{T_2!}=\sum_{N+M=T_2}
\frac{E_{3,2}^{N}}{N!}\otimes \frac{E_{3,2}^{M}}{M!}.
$$

One obtains the following expression:
\begin{align}
\label{kipelov}
\sum \!\const
\cdot i! j! o! e! 
\cdot
\frac{a_3^{i}}{i!}\frac{a_{1,2}^{j}}
{j!}
\frac{b_{3}^{o}}{o!}
\frac{b_{1,2}^{e}}{e!}\!\cdot\! Z_{N_1\!,N_2\!,P}^{\gamma^t\!,u+q\!+k_1\!+t}
\!Z_{M_1\!,M_2\!,P}^{\delta^s\!,g_{\tau}\!+q\!+k_3\!+s} \!\! \cdot\!
F_{\epsilon}(a)\cdot G_{\varepsilon}(b).
\end{align}

The summation is taken over the partitions and integer parameters form the previous steps and also over new non-negative integer parameters   $r$, $l$.
Here $const$  is a constant from  \eqref{dddd} and
\begin{align*}
\begin{cases}
i=u+q+k_1+t_{}+N_1+r,\quad  j=v+q+\gamma-\psi'+t_{}-N_2+r,\\
o=g+q+k_3+s_{}+M_1+l,\quad  e=h+q+\delta-\varphi'+s_{}-M_2+l,
\\[2mm] 
\epsilon_1=\varphi_1+\omega_1-\psi_1+\alpha-k_1-\varphi',\quad
\epsilon_2=\varphi_2+\omega_2+\psi_1-t-N_1-r,\quad \\ \epsilon_{1,3}=\psi_2+\psi_1-t+N_2-r,\quad \epsilon_{2,3}=0,
\\[2mm] 
\vartheta_1=\psi_1+\omega_2-\varphi_1+\beta-k_3-\psi',\quad
\vartheta_2=\psi_2+\omega_1+\varphi_1-s-M_1-l,\quad \\ \vartheta_{1,3}=\varphi_2+\varphi_1-s+M_2-l,\quad \vartheta_{2,3}=0.
\end{cases}
\end{align*}

\noindent{\bf 5)} Apply to  \eqref{kipelov} the operator  $\frac{E_{2,1}^{S}}{S!}$, which acts onto a tensor product as follows:
$$
\frac{E_{2,1}^{S}}{S!}=\sum_{H+J=S} \frac{E_{2,1}^{H}}{H!}\otimes
\frac{E_{2,1}^{J}}{J!}.
$$
As a result one obtains a linear combination of products
\begin{equation}\label{prdct}\frac{a_3^{u}a_{1,2}^v}{u!v!}F_{\rho}(a)\cdot
\frac{b_3^{g}b_{1,2}^{h}}{g!h!}G_{\varrho}(b),\end{equation} где
\begin{equation}
\begin{cases}
u=\varphi_3+q_2+q+k_1+t+N_1+r,\quad  \\ v=\psi_3+q_3+q+\gamma-\psi'+t+N_2+r,\quad
\\
N=N_1+N_2,
\\
g=\psi_3+q_3+q+k_3+s+M_1+l,\quad  \\ h=\varphi_3+q_2+q+\delta-\varphi'+s+M_2+l,\quad  \\  M=M_1+M_2,
\\[2mm]
\rho_1=\varphi_1+\omega_1-\psi_1+\alpha-k_1+\varphi'-H,\quad
\\
\rho_2=\varphi_2+\omega_2+\psi_1-t-N_1-r+H,\quad
\\ \rho_{1,3}=\psi_2+\psi_1-t+N_2-r\quad \rho_{2,3}=0,\\[2mm]
\varrho_1=\psi_1+\omega_2-\varphi_1+\beta-k_3+\psi'-J,\quad  \\ \varrho_2=\psi_2+\omega_1+\varphi_1-s-M_1-l+J,\quad
\\ \varrho_{1,3}=\varphi_2+\varphi_1-s+M_2-l,
\quad \varrho_{2,3}=0.\\
\end{cases}
\end{equation}

The factor $\frac{a_3^{u}a_{1,2}^v}{u!v!}F_{\rho}(a)$ corresponds to the diagram
\begin{align}
\begin{split}
\label{vsd1}
&\begin{pmatrix} m_1 && m_2 &&0 \\& m_1-t_1 && m_2-t_2 \\&&m_1-t_1-s
\end{pmatrix},\quad
\\&
\begin{cases}
t_1=(\varphi_3+q_2+q+k_1)+(t+N_1+r)\\
t_2=(\psi_2+\psi_1)-(t-N_2+r)\\
s=(\varphi_2+\omega_2+\psi_1)-(t+N_1+r)+H.
\end{cases}
\end{split}
\end{align}

The factor  $\frac{b_3^{g}b_{1,2}^{h}}{g!h!}G_{\varrho}(b)$ corresponds to the diagram
\begin{align}
\begin{split}
\label{vsd2}
&\begin{pmatrix} \bar{m}_1 && \bar{m}_2 &&0 \\& \bar{m}_1 -\bar{t}_1&& \bar{m}_2-\bar{t}_2 \\&&\bar{m}_1-\bar{t}_1-\bar{s}
\end{pmatrix},\quad
\\&
\begin{cases}
\bar{t}_1=(\psi_3+q_3+q+k_3)+(s+M_1+l)\\
\bar{t}_2=(\varphi_2+\varphi_1)-(s-M_2+l)\\
\bar{s}=(\psi_2+\omega_1+\varphi_1)-(s+M_1+l)+J.
\end{cases}
\end{split}
\end{align}

The coefficient at the product  \eqref{prdct} equals to
\begin{align}
\begin{split}
\label{coefcg}
&d_{k_1+k_3,\varphi'+\psi'}^{M_1-M_3}\frac{\lambda!}{\varphi!\varphi'!}
\frac{\mu!}{\psi!\psi'!}(-1)^{t+s+r+l}i! j! o! e!
\\
&\times \frac{  X_{\psi_i,\varphi_i,\omega_i}^{\theta,\vartheta,q_i}Y^{\theta,\alpha-k_1-\varphi'}_t
	Y^{\vartheta,\beta-k_3-\psi'}_s Z_{N_1,N_2,P}^{\gamma^t,u+q+k_1+t}
	Z_{M_1,M_2,P}^{\delta^s,g+q+k_3+s}  }{
	(\alpha-k_1-\varphi)!(\beta-k_3-\psi)!(\gamma-\psi)!(\delta-\varphi)!k_1!k_3!}
\\
&\times (u+q+k_1+t_{})!(v+q+\gamma-\psi'+t_{})!
\\
&\times
(g+q+k_3+s_{})!(h+q_{}+\delta-\varphi'+s_{})!\,.
\end{split}
\end{align}

The calculation of this expression should  be  done as follows.  Initially we are given  $\alpha,\beta,\gamma,\delta,\omega,\varphi,\psi$ and  \eqref{Diag0}. From  \eqref{Diag0}  we know  $M_1-K_1$, $M_2-K_2$, $K_1-S$.
\begin{enumerate}
	\item Fix  a partition  $k_1+k_3+\varphi'+\psi'=M_1-K_1$ and then calculate
	$d_{k_1+k_3,\varphi'+\psi'}^{M_1-M_3}$.
	\item Using  $\omega$, $\lambda=\varphi+\varphi',\mu=\psi+\psi'$,
	calculate $\theta$, $\vartheta$ and
	$X_{\psi_i,\varphi_i,\omega_i}^{\theta,\vartheta,q_i}$ according to the  Proposition  \ref{px}.
	
	\item Calculate  $\gamma^t$, $\delta^t$, $Y^{\theta,\alpha-k_1-\varphi'}_t$, $
	Y^{\vartheta,\beta-k_3-\psi'}_s$ according to the  Proposition   \ref{py}.
	\item Fix a partition  $M_2-K_2=N_1+N_2+M_1+M_2$.Using the  Proposition \ref{pz}, find $Z_{N_1,N_2,P}^{\gamma^t,u+q+k_1+t}$ и
	$Z_{M_1,M_2,P}^{\delta^s,g+q+k_3+s}$ , $i$, $j$, $o$, $e$.
\end{enumerate}

\subsection{The Clebsh-Gordan coefficients}
\label{cg}

Let us give the final answer in a simpler form. Take in
\eqref{rcg} a representation with the highest vector
\eqref{stv1}.  Take in this representation a vector corresponding to the diagram
\begin{equation}
\label{Diag}
\begin{pmatrix}
M_1&&M_2&&M_3\\
&M_1-T_1 && M_2-T_2\\
&&\quad M_1-T_1-S
\end{pmatrix}.
\end{equation}

In the previous Section we fixed  partitions
${T_1\!=\!k_1\!\!+\!k_3\!+\!\varphi'\!\!+\!\psi'}\!$, $T_2\!=\!N_1+N_2+M_1+M_2$, $S=J+H$,
integers  $q,r,s,r,l\in\mathbb{Z}_{\geq 0}$, and the partitions
$q=:q_1+q_2+q_3$ and
$\varphi'+\varphi-q-q_3=:\varphi_1+\varphi_2+\varphi_3$,
$\psi'+\psi-q-q_2=:\psi_1+\psi_2+\psi_3$.

Let us give another description of parameters involved in formulas
\eqref{vsd1}, \eqref{vsd2}, \eqref{coefcg}.   For the given diagram
\eqref{Diag}  we fix  partitions
\begin{align}
\begin{split}
\label{r1}
T_1&=:T'_1+T''_1+T'''_1+T''''_1,\quad  T_2=:T'_2+T''_2,
\\
T''_1+T''_2+\omega&=:L_1+L_2, \quad S=:S'+S'',
\\
&\text{   fix arbitrary  }A,B\in\mathbb{Z}_{\geq 0}.
\end{split}
\end{align}

\begin{theorem}
	\label{mnt11}
	Take a  $\mathfrak{gl}_3$-highest vector in    $U\otimes V$ of type
	$f(\omega,\varphi,\psi,0)$. Take a vector  \eqref{Diag} in the corresponding representation  $W^{\alpha}$.   Then the vector
	\eqref{Diag}  is a linear combination of tensor products of vectors
		\begin{align*}
		\begin{pmatrix}m_1 && m_2 &&0\\
		&m_1-T'_1-A && m_2-T'''_1-T'_2+A\\
		&& m_1-T'_1-S'-L_1 \end{pmatrix}
		\end{align*}
	и
		\begin{align*}
		\begin{pmatrix}\bar{m}_1 && \bar{m}_2 && 0\\
		&\bar{m}_1-T''_1-B &&\bar{m}_2 -T''''_1-T''_2+B\\
		&& \bar{m}_1-T''_1-S''-L_2 \end{pmatrix}.
		\end{align*}
\end{theorem}

To write a coefficient we fix partitions
\begin{align}
\begin{split}
\label{r2}
& T'_1=:k_1+\varphi_3+q+q_2,\quad  T''_1=:\varphi_1+\varphi_2,\\
& T''_1=:k_3+\psi_3+q+q_3,\quad  T''''_1=:\psi_1+\psi_2,\\
&\omega_2:=L_1-\varphi_2-\psi_1,\quad  \omega_1:=L_2-\psi_1-\varphi_2 \\
&A=:t+r+U,\quad  B=:s+l+V.
\end{split}
\end{align}
Introduce notations
\begin{align}
\begin{split}
\label{r3}
&q_1:=q-q_2-q_3,\quad
H:=S',\quad  J=S'',\quad \\&  N_1=U,\quad
N_2=T'_2-U,\quad M_1=V,\quad
M_2=T''_2-V.
\end{split}
\end{align}

The calculations \textup,  made in the previous Section \textup, give the following theorem.
\begin{theorem}
	\label{mnt12}
	The coefficient at the tensor product form the Theorem  \textup{\ref{mnt11}}  equals to the sum of expressions   \eqref{coefcg} over all partitions  \eqref{r2}\textup, \eqref{r3}\textup,where symbols  $X,Y,Z$  о are defined in Propositions
	\textup{\ref{px}, \ref{py}, \ref{pz},}  and parameters in the arguments  $X,Y,Z$ are defined by partitions    \eqref{r1}\textup, \eqref{r2}\textup, \eqref{r3}.
\end{theorem}

\section{The case of the vector  $f(0,\varphi,\psi,\theta)$}
\label{sl2}
\subsection{Two viewpoints on the contragradient representation }
Take a representation  $V$ with the highest vector
\begin{equation}\label{s1}a_1^{m_1-m_2}a_{1,2}^{m_2}.
\end{equation}

The exists an invariant scalar product on the space  $V$ \break
$(\,\cdot\,,\,\cdot\,),$ such that the Gelfand-Tselin base is orthogonal with respect to this base.  Using this scalar product one can identify $V$
and  $V^*:$   $v\leftrightarrow(v,\,\cdot\,)$.  The action of the algebra ontoа  $V^*$ is the following\textup:
\begin{equation}
g:  \,(v,\,\cdot\,)\mapsto (-g^tv,\,\cdot\,).
\end{equation}
The obtained representation is called contragradient.
\begin{prop}
	The base in  $V^*,$  dual to the  Gelfand-Tsetlin base in  $V,$  is a base proportional to the Gelfand-Tsetlin base in   $V^*$.
\end{prop}
\proof
Let
$z_+\in \mathfrak{gl}_3$ be an element of the  algebra corresponding to a positive root.  If  $v_0\in V^*$ is the highest vector then
$$
0=(-z_{+}^tv_0,v)=(v_0,z_{+}v),
$$
i.e. $v_0$  is orthogonal to all vectors of type  $z_{+}v$.  But span of the vectors of type $z_{+}v$  is a span of all non-lowest Gelfand-Tsetlin  vectors. An orthogonal complement to this span is generated by the lowest vector. Hence  $v_0$  under the identification of  $V^*$  and   $V$  is mapped to the vector proportional to the lowest vector.  Considering elements  $z_+\in
\mathfrak{gl}_2\subset \mathfrak{gl}_3$ and $z_+\in
\mathfrak{gl}_1\subset \mathfrak{gl}_2\subset \mathfrak{gl}_3$,  one obtains that a base dual to the Gelfand-Tselin base is proportional to the Gelfand-Tselin base in $V^*$.
\endproof

Thus  $V^*$ is a representation with a highest vector proportional to
\begin{equation}
\label{aa3} a_{3}^{(m_1-m_2)}a_{2,3}^{m_2}.
\end{equation}

If one uses an ordinary realization of the contragradient representation then it is realized as a representation with the highest vector proportional to
\begin{equation*}a_{1,2}^{-m_2}a_{1,2,3}^{-m_1}.
\end{equation*}
Multiply this representation onto a representation with the highest vector ~$a_{1,2,3}^{m_1}$.
Then one obtains a representation with the highest weight  $[m_2,m_1-m_2,0]$ and the highest vector
\begin{equation}\label{s2}a_1^{m_2}a_{1,2}^{m_1-m_2}.
\end{equation}

Consider a mapping
\begin{equation}
\label{sopost} a_3\leftrightarrow a_{1,2},\quad a_1\leftrightarrow
a_{2,3},\quad a_2\leftrightarrow -a_{1,3},
\end{equation}
which a bijection between the space of the contragradient representation (i.e. a representation with the highest vector
\eqref{aa3}) into the space of a representation with highest vector \eqref{s2}.
This mapping conjugates the actions
$$
v\mapsto -E_{i,j}^tv\,\,\text{  и  }\,\,w\mapsto (m_1\delta_{i,j}+E_{i,j})w.
$$
\begin{prop}
	Under the mapping  \eqref{sopost}  the Gelfand-Tselin vector is mapped into a Gelfand-Tselin vector up to a sign
		\begin{equation}
		\label{sopd}
		\begin{split}
		\!\begin{pmatrix}m_1 && m_2 &&0\\ &\!\! m_1\!-\!t_1 && \!\!m_2\!-\!t_2\\&& \!\!m_1\!-\!t_1\!-\!s
		\end{pmatrix}
		\mapsto (\!-\!1)^{s+t_2}  \!\begin{pmatrix}m_1 &&\!\! m_1\!-\!m_2 &&0\\ &\!\! m_1\!-\!m_2\!+\!t_2 &&\!\! t_1\\&& t_1\!+\!s
		\end{pmatrix}.
		\end{split}
		\end{equation}
\end{prop}

\begin{proof}
	To the diagram on the left side of  \eqref{sopd},
	there corresponds a function
	$$\frac{a_3^{t_1}}{t_1!}\frac{a_{1,2}^{t_2}}{t_2!}F_{(m_1-m_2-t_1-s,s,t_2,0)}(a_1,a_2,a_{1,3},a_{2,3}).$$  Under the mapping   \eqref{sopost}  this function is transforms to
	\begin{align*}
	&\frac{a_{1,2}^{t_1}}{t_1!}\frac{a_{3}^{t_2}}{t_2!}\cdot F_{(m_1-m_2-t_1-s,s,t_2,0)}(a_{2,3},-a_{1,3},-a_2,a_{1}).
	\end{align*}
	Using a relation between  a $\Gamma$-series and a Gauss' hypergeometric functions  $F_{2,1}$,  one obtains
	\begin{align*}
	&const\cdot \frac{a_{1,2}^{t_1}}{t_1!}\frac{a_{3}^{t_2}}{t_2!} a_1^{m_1-m_2-t_2-s}(-a_2)^{s}(-a_{1,3})^{t_2}F_{2,1}\bigg(\dots;\frac{(-a_1)(-a_{2,3})}{a_2a_{1,3}}\bigg)
	\\
	&= (-1)^{s+t_2}\frac{a_{3}^{t_2}}{t_2!} \frac{a_{1,2}^{t_1}}{t_1!}F_{(m_1-m_2-t_1-s,s,t_2,0)}(a_{2,3},a_{1,3},a_2,a_{1})
	\\
	&=(-1)^{s+t_2}\frac{a_{3}^{t_2}}{t_2!} \frac{a_{1,2}^{t_1}}{t_1!}F_{(m_2-m_1+t_1+s,m_1-m_2-t_1-s+t_2,m_1-m_2-t_1,0)}(a_1,a_2,a_{1,3},a_{2,3}).
	\end{align*}
	This function corresponds to a diagram on the right side in  \eqref{sopd}.
\end{proof}

The  lowering operators  $\nabla_{3,1}$, $\nabla_{3,2}$,
$\nabla_{2,1}$, acting onto a diagram on the left are conjugated to the operators
$\nabla_{1,3}=-E_{1,3}+\frac{1}{-E_{1,1}-E_{2,2}+1}E_{2,3}E_{1,2}$,
$\nabla_{2,3}=E_{2,3}$, $\nabla_{1,2}=E_{1,2}$ acting on a diagram on the right. We call them the     ``raising operators''. Applying these operstors to the lowest vector one can  obtain a vector corresponding to an arbitrary diagram.

Note that under the mapping  \eqref{sopost} (and analogous mapping for   $b$)
one has
\begin{equation}
(aabb)_{2313}^{\theta}(abb)^{\phi}(aab)^{\psi}\mapsto (-1)^{\theta}
(ab)_{12}^{\theta}(aab)^{\phi}(abb)^{\psi}.
\end{equation}

According to previous considerations we come to the following conclusion.  Take a tensor product of representations with the highest weights  $[m_1,m_1-m_2,0]$ and  $[\bar{m}_1,\bar{m}_1-\bar{m}_2,0]$.  Split it into a sum of irreducibles.   Take a summand with the   {\it lowest  } vector of type $f(\theta,\varphi,\psi,0)$.  Apply the mapping   \eqref{sopost}  to all this construction.  We obtain a tensor product of representations with the highest weights  $[m_1,m_2,0]$ and  $[\bar{m}_1,\bar{m}_2,0]$. This representation is splitted into a sum of irreducibles.  A chosen lowest vector of type  $f(\theta,\varphi,\psi,0)$  is mapped to the highest veftor of type  $f(0,\psi,\varphi,\theta)$,  multiplied by $(-1)^{\theta}$.   Since the actions of the lowing and of the raising operators are conjugated we conclude that  the Gelfand-Tsetlin base vector in the chosen irreducible summand is mapped to a Gelfand-Tselin base vector in the corresponding irreducible summand up to a sign  $(-1)^{\theta+S+T_2}$.

\subsection{The Clebsh-Gordan coefficients} \label{cg2}
Consider a tensor product of representations with the highest weights   $[m_1,m_2,0]$ and
$[\bar{m}_1,\bar{m}_2,0]$,  split it into a sum  of irreducibles and take a summand with the  lowest vector  of type
$$
\dots(aabb)_{2,3,1,3}^{\theta}(aab)^{\varphi}(abb)^{\psi}.
$$

Take a vector in this representation corresponding to a diagram
\begin{equation}
\begin{split}
\label{Diag1}
\xymatrix{ M_{1}  && M_2  && 0\\
	&M_2+T_2 \ar[ur]_{T_2}    && T_1\ar[ur]_{T_2}  \\&&T_1+S \ar[ur]_{S}
}
\end{split}
\end{equation}

Fis a decomposition  \eqref{r1}. One has a Theorem.
\begin{theorem}
	\label{mnt21}
	Fix a $\mathfrak{gl}_3$-highest vector in $U\otimes V$ of type  $f(0,\varphi,\psi,\theta)$. Take a vector \eqref{Diag1}  in the corresponding representation $W^{\alpha}$.   It is written as a linear combination of tensor product of diagrams
		\begin{align*}
		\begin{pmatrix}m_1 && m_2 &&0\\
		&m_2+T'''_1+T'_2-A && T'_1+A\\
		&&T'_1+S'+L_1 \end{pmatrix}
		\end{align*}
	and
		\begin{align*}
		\begin{pmatrix}\bar{m}_1 && \bar{m}_2 && 0\\
		&\bar{m}_2+T''''_1+T''_2+B &&T''_1+B\\
		&& T''_1+S''+L_2\end{pmatrix}.
		\end{align*}
\end{theorem}

Fix decompositions  \eqref{r2}, \eqref{r3}. Then the following theorem takes place.
\begin{theorem}
	\label{mnt22}
	A copefficient at the product form the Theorem   \textup{\ref{mnt21}} equals to the sum of expressions  \eqref{coefcg} over partitions   \eqref{r2}, \eqref{r3}\textup, multiplied by the sign \textup:
	$$
	(-1)^{\theta+(T_2+S)+(T'''_1+T'_2+S'+L_1)+(T''''_1+T''_2+S''+L_2)}
	=(-1)^{T'''_1+T''''_1+T''_1+T''_2+\omega+\theta}.
	$$
\end{theorem}

Note that in the case $\theta=\omega=0$ the highest vector is both of the first and of the second type. But theorems  \ref{mnt11}, \ref{mnt12} and  \ref{mnt21}, \ref{mnt22}  give different answers since different approaches are used.

\section{The derivation of statements about  a $\Gamma$-series}
\label{pr}
\subsection{Statements associated with the homogeneity equation}

In the GKZ system  \eqref{gkzs}  the second, the third, the forth equations describe the homogeneity property of the function   $F_{\gamma}$. Using them let us prove the following equation.
\begin{prop}
	\label{f13}
	\begin{align*}
	&a_{1,3}F_{\gamma}=X_1F_{\gamma^1}+(a_{1}a_{2,3}-a_{2}a_{1,3})X_2F_{\gamma^2},
	\end{align*}
	where  $X_1,X_2$ ---  are some constants\textup,  $\gamma^1,$ $\gamma^2$ ---
	are new parameters of a  $\Gamma$-series.
\end{prop}

\proof
Using the homogeneity equation from the GKZ system one can obtain:
\begin{align*}
&\begin{cases}a_1F_{\gamma}=-a_{1,3}F_{\gamma+e_1-e_{1,3}}+\const\cdot F_{\gamma+e_1},\\
a_{2,3}F_{\gamma}=-a_{1,3}F_{\gamma+e_{2,3}-e_{1,3}}+\const\cdot F_{\gamma+e_{2,3}} \end{cases}
\\&
\quad\Rightarrow a_{1}a_{2,3}F_{\gamma}=
a_{1,3}^2F_{\gamma+e_{2,3}-e_{1,3}+e_1-e_{1,3}}+\const\cdot
a_{1,3}F_{\gamma+e_1+e_{2,3}-e_{1,3}}
\\&\qquad+const\cdot
F_{\gamma+e_1+e_{2,3}}.
\end{align*}
Also
\begin{align*}
a_{2}F_{\gamma}&=-a_1F_{\gamma+e_2-e_1}+\const\cdot F_{\gamma+e_2}\quad
\\
&\Rightarrow a_{2}F_{\gamma}=a_{1,3}F_{\gamma+e_2-e_1+e_1-e_{1,3}}+\const\cdot
F_{\gamma+e_2-e_1+e_1}+\const\cdot F_{\gamma+e_2}
\\&
=a_{1,3}F_{\gamma+e_2-e_{1,3}}+\const\cdot
F_{\gamma+e_2}
\\&
\Rightarrow\,\,a_{2}a_{1,3}F_{\gamma}=a_{1,3}^2F_{\gamma+e_2-e_{1,3}}+\const\cdot
a_{1,3}F_{\gamma+e_2}.
\end{align*}
Let us use
$
F_{\gamma+e_{2,3}-e_{1,3}+e_1-e_{1,3}}=F_{\gamma+e_2-e_{1,3}},\quad F_{\gamma+e_1+e_{2,3}-e_{1,3}}=F_{\gamma+e_2}.
$

Hence,
\begin{align}
\begin{split}
\label{a13}
&(a_{1}a_{2,3}-a_{2}a_{1,3})F_{\gamma}=\const\cdot a_{1,3}F_{\gamma+e_2}+\const\cdot F_{\gamma+e_1+e_{2,3}}.
\end{split}
\end{align}
\endproof
Analogously the following propositions can be proved.
\begin{prop}
	\begin{align*}
	&a_{2}F_{\gamma}=X_1  F_{\gamma^1}
	+(a_{1}a_{2,3}-a_{2}a_{1,3})X_2F_{\gamma^2},
	\end{align*}
	where   $X_1,X_2$ ---  are some constants  \textup(other then in  \textup previous Proposition\textup).
\end{prop}

\begin{prop}
	\begin{align*}
	&a_{1}F_{\gamma,B}=X_1  F_{\gamma^1}
	+(a_{1}a_{2,3}-a_{2}a_{1,3})X_2F_{\gamma^2}.
	\end{align*}
\end{prop}

\begin{prop}
	\begin{align*}
	&a_{2}F_{\gamma}= +X_1F_{\gamma^1}+X_2 (a_{1}a_{2,3}-a_{2}a_{1,3})  F_{\gamma^2}.
	\end{align*}
\end{prop}

From these Propositions one gets a Lemma.
\begin{lemma}
	For arbitrary $u,v,w,t$  one has
	\begin{equation}
	\label{l1}
	a_1^{u}a_2^{v}a_{1,3}^{w}a_{2,3}^tF_{\gamma}
	=\sum_{\alpha}
	X_{\alpha}(a_2a_{1,3}-a_1a_{2,3})^{p_{\alpha}}F_{\gamma^{\alpha}}.
	\end{equation}
	
\end{lemma}

Since for   $\gamma=(0,0,0,0)$ one has  $F_{\gamma}\equiv 1$,  then the following statement takes place.
\begin{lemma}
	For arbitrary  $u,v,w,t$ one has
	\begin{equation}
	\label{l2}
	a_1^{u}a_2^{v}a_{1,3}^{w}a_{2,3}^t
	=\sum_{\alpha}
	X_{\alpha}(a_2a_{1,3}-a_1a_{2,3})^{p_{\alpha}}F_{\gamma^{\alpha}}.
	\end{equation}
\end{lemma}

\subsection{Statement associated with the hypergeometric operator}
Let us introduce a notation for the first operator from \eqref{gkzs}:
\begin{equation}
O:=\frac{\partial^2}{\partial a_1\partial
	a_{2,3}}-\frac{\partial^2}{\partial a_2\partial
	a_{1,3}}.
\end{equation}

\subsubsection{Auxiliary statements}
\begin{prop}
	\label{pre1}
	\begin{equation}
	O\big( (a_1a_{2,3}\!-\!a_{2}a_{1,3})^kF_{\gamma}
	\big)\!=\!(k(k\!+\!1)\!+\!k(\gamma_1\!+\!\gamma_2\!+\!\gamma_{1,3}))(a_1a_{2,3}\!-\!a_{2}a_{1,3})^{k\!-\!1}\!\cdot\!
	F_{\gamma}.
	\end{equation}
\end{prop}
\proof
One has
\begin{equation}
\label{oaa} O\big( (a_1a_{2,3}-a_{2}a_{1,3})^k\big)=
k(k+1)(a_1a_{2,3}-a_{2}a_{1,3})^{k-1}.
\end{equation}
Now let us find
\begin{align*}
&\Big( \frac{\partial}{\partial a_1}(a_1a_{2,3}-a_{2}a_{1,3})^k\Big)
\frac{\partial }{\partial a_{2,3}}F_{\gamma}+\Big(
\frac{\partial}{\partial a_{2,3}}(a_1a_{2,3}-a_{2}a_{1,3})^k\Big)
\frac{\partial }{\partial a_{1}}F_{\gamma}
\\
&-\Big( \frac{\partial}{\partial
	a_2}(a_1a_{2,3}-a_{2}a_{1,3})^k\Big) \frac{\partial }{\partial
	a_{1,3}}F_{\gamma}-\Big( \frac{\partial}{\partial
	a_{1,3}}(a_1a_{2,3}-a_{2}a_{1,3})^k\Big) \frac{\partial }{\partial
	a_{2}}F_{\gamma}
\\
&=k(a_1a_{2,3}-a_{2}a_{1,3})^{k-1}\Big( a_1\frac{\partial }{\partial
	a_{1}}+a_2\frac{\partial }{\partial a_{2}}+a_{1,3}\frac{\partial
}{\partial a_{1,3}}+a_{2,3}\frac{\partial }{\partial a_{2,3}}
\Big)F_{\gamma}
\\
&=
k(a_1a_{2,3}-a_{2}a_{1,3})^{k-1}(\gamma_1+\gamma_2+\gamma_{1,3})F_{\gamma}.
\end{align*}

Using  $O F_{\gamma}=0$,  one gets
\begin{equation*}
\begin{split}
O\big( (a_1a_{2,3}&-a_{2}a_{1,3})^kF_{\gamma}
\big)
\\
&=(k(k+1)+k(\gamma_1+\gamma_2+\gamma_{1,3}))(a_1a_{2,3}-a_{2}a_{1,3})^{k-1}F_{\gamma}.\qedhere
\end{split}
\end{equation*}
\endproof

\subsubsection{The scheme of derivation of the main statements}
We are going to use the following idea. If one knows that a function  $f(z)$
can be presented as a power series then the coefficients of this series can be calculated by the following ruler
\begin{equation}
\label{fz} f(z)=c_0+c_1z+c_2z^2+c_3z^3\dots\,\,\Rightarrow\,\,
c_0=f(0),\quad
c_1=\frac{d}{dz}f(z)\mid_{z=0},\quad \dots\,.
\end{equation}
Let us use the ruler  \eqref{fz}, where analogs of  $\frac{d}{dz}$ and  $\mid_{z=0}$ are the following
\begin{equation*}
\frac{d}{dz}\mapsto O=\frac{\partial^2}{\partial a_1\partial
	a_{2,3}}-\frac{\partial^2}{\partial a_2\partial a_{1,3}},
\quad .\mid_{z=0}\mapsto.\mid_{a_1a_{2,3}=a_2a_{1,3}}.
\end{equation*}

\subsubsection{The first main statement}
The fact that the relation   \eqref{rel1}  takes place follows from \eqref{l1}.
Let us prove the statement \ref{py},  which gives a formula for the coefficients of this relation.

Let us find terms with $p_{\tau}=p$.  For this purpose let us apply to both parts of
\eqref{rel1}  the operator   $O^p$, and then let us make a substitution
$.\mid_{a_{1}a_{2,3}=a_{2}a_{1,3}}$.

\noindent{\bf 1)} On the left one has
\begin{equation}
O^p(a_1^uF_{\gamma})=ua_1^{u-1}F_{\gamma-e_{2,3}}=\frac{\Gamma(u)}{\Gamma(u-p)}a_1^{u-p}F_{\gamma+pe_1-pe_2-pe_{1,3}}.
\end{equation}
After the substitution  $.\mid_{a_{1}a_{2,3}-a_{2}a_{1,3}}$ one obtaines
\begin{equation}
\frac{\Gamma(u)}{\Gamma(u-p)}a_1^{u+\gamma_1}a_2^{\gamma_2-p}a_{1,3}^{\gamma_{1,3}-p}F_{\gamma+pe_1-pe_2-pe_{1,3}}({\bf
	1}).
\end{equation}

\noindent{\bf 2)} On the right after application of  $O$  and the substitution
$.\mid_{a_{1}a_{2,3}=a_{2}a_{1,3}}$  there remains only the summands
\eqref{rel1} for which $p_{\tau}=p$. According to the Proposition   \ref{pre1}  at every summand there appeares a coefficient
$$
\Pi_{p,\gamma^{\tau}}=\prod_{t=1}^p\big(t^2+t(\gamma^{\tau}_1+\gamma^{\tau}_2+\gamma^{\tau}_{1,3}+\gamma^{\tau}_{2,3})\big).
$$

Hence in \eqref{rel1}  for terms with $p_{\tau}=p$  one  has
\begin{align*}
&\gamma^{\tau}_1=\gamma_1+u,\quad \gamma^{\tau}_2=\gamma_2-p,\\
&\gamma^{\tau}_{1,3}=\gamma_{1,3}-p,\quad \gamma^{\tau}_{2,3}=0,\\
&Y_{\gamma^{\tau}}=\frac{\Gamma(u)}{\Gamma(u-k)}\cdot\frac{F_{\gamma}({\bf
		1})}{F_{\gamma^{\tau}}({\bf 1})}\cdot
\frac{1}{\Pi_{p,\gamma^{\tau}}}.
\end{align*}

\subsubsection{The second main statement}

The fact that the relation  \eqref{rel2} takes place follows from
\eqref{l2}. Let us prove the Proposition \ref{px}.

Let us  use the principle  \eqref{fz}.  We need to apply the oparators
$O_a^{p}O_b^q$, where
\begin{align*}
O_a=\frac{\partial^2}{\partial a_1\partial
	a_{2,3}}-\frac{\partial^2}{\partial a_2\partial
	a_{1,3}},\quad O_b=\frac{\partial^2}{\partial b_1\partial
	b_{2,3}}-\frac{\partial^2}{\partial b_2\partial b_{1,3}}.
\end{align*}

To find summands with  $p_{\tau}=p,q_{\tau}=q$, let us apply $O_a^pO_b^q$ and make substitutions
$.\mid_{a_{1}a_{2,3}=a_{2}a_{1,3},\,\,b_{1}b_{2,3}=b_{2}b_{1,3}}$ on the left and on the right.

On the right in  \eqref{rel2} we obtain summands with
$p_{\tau}=p,q_{\tau}=q$, transformed to
\begin{align}
\begin{split}
\label{pravo}
&\Pi_{p,\theta^{\tau}}\cdot \Pi_{p,\vartheta^{\tau}}\cdot
X_{\tau}a_3^{u}a_{1,2}^{v}
b_{3}^{g}b_{1,2}^{h}
\cdot a_1^{\theta_1}a_2^{\theta_2}
a_{1,3}^{\theta_{1,3}}F_{\theta}({\bf 1})\cdot
b_1^{\vartheta_1}b_2^{\vartheta_2}b_{1,3}^{\vartheta_{1,3}} G_{\vartheta}({\bf 1}).
\end{split}
\end{align}
Consider the left side. Since $O_a^p$ and  $O_b^q$  commute we can apply them in an arbitrary order. Introduce a notation
\begin{equation}
\label{fpsi}
f(\lambda,\mu,\omega):=\frac{(abb)^{\lambda}}{\lambda!}\frac{(aab)^{\mu}}{\mu!}\frac{(ab)_{1,2}^{\omega}}{\omega!}.
\end{equation}
This expression equals to
\begin{align}
\begin{split}
\label{abbbba}
&
\frac{(a_1b_{2,3}-a_2b_{1,3}+a_3b_{1,2})^{\lambda}}{\lambda!}
\frac{(b_1a_{2,3}-b_2a_{1,3}+b_3a_{1,2})^{\mu}}{\mu!}\frac{(a_1b_2-a_2b_1)^{\omega}}{\omega!}
\\
&=
\sum_{}
a_1^{\varphi_1+\omega_1}a_2^{\varphi_2+\omega_2}a_3^{\varphi_3}\cdot
a_{1,2}^{\psi_3}a_{1,3}^{\psi_2}
a_{2,3}^{\psi_1} \cdot
b_1^{\psi_1+\omega_2}b_2^{\psi_2+\omega_1}b_3^{\psi_3}\cdot
b_{1,2}^{\varphi_3}b_{1,3}^{\varphi_2}b_{2,3}^{\varphi_1}
\\
&\qquad\qquad\times
\frac{(-1)^{\varphi_2+\psi_2+\omega_2}}{\varphi_1!\varphi_2!\varphi_3!\psi_1!\psi_2!\psi_3!\omega_1!\omega_2!}.
\end{split}
\end{align}
The summation is taken over all indices  $\varphi_i,\psi_i,\omega_i$,
such that
$$
\begin{cases}\varphi_1+\varphi_2+\varphi_3=\lambda,\\
\psi_1+\psi_2+\psi_3=\mu,\\ \omega_1+\omega_2=\omega.
\end{cases}
$$
Using   \eqref{abbbba}, we obtain that application of  $O_b$ to
\eqref{fpsi}  gives
\begin{multline*}
a_{1}a_{2,3}f(\lambda-1,\mu-1,\omega)
+a_1(-a_2)f(\lambda-1,\mu,\omega-1)
\\
-(-a_{2})(-a_{1,3})f(\lambda-1,\mu-1,\omega)
-(-a_2)a_1f(\lambda-1,\mu,\omega-1)
\\
=(a_1a_{2,3}-a_2a_{1,3})f(\lambda-1,\mu-1,\omega).
\end{multline*}
Thus an application of  $O_b^q$ to  \eqref{fpsi} gives 
\begin{align}
\label{oar} &(a_1a_{2,3}-a_2a_{1,3})^qf(\lambda-q,\mu-q,\omega).
\end{align}
Now apply  $O_a$  to this expression. Using   \eqref{oaa}, one obtains that
$O_a((a_1a_{2,3}-a_2a_{1,3})^qf(\lambda,\mu,\omega))$ equals
\begin{equation}
\label{obr}
\begin{split}
q^2(a_1a_{2,3}&-a_2a_{1,3})^{q-1}f(\lambda,\mu,\omega)
\\
&\quad+(b_2b_{1,3}-b_1b_{2,3})(a_1a_{2,3}-a_2a_{1,3})^{q}f(\lambda-1,\mu-1,\omega)
\\
&\quad+q(a_1a_{2,3}-a_2a_{1,3})^{q-1}\Big (
a_{2,3}b_1f(\lambda,\mu-1,\omega)
\\
&\quad+a_{1}b_{2,3}f(\lambda-1,\mu,\omega)
+a_{1}b_2f(\lambda,\mu,\omega-1)
\\
&\quad-(-a_{1,3})(-b_2)f(\lambda,\mu-1,\omega)
\\
&\quad-(-a_{2})(-b_{1,3})f(\lambda-1,\mu,\omega)
-(-a_{2})(-b_1)f(\lambda,\mu,\omega-1) \Big)
\\
&=q^2(a_1a_{2,3}-a_2a_{1,3})^{q-1}f(\lambda,\mu,\omega)
\\
&\quad+(b_2b_{1,3}-b_1b_{2,3})(a_1a_{2,3}-a_2a_{1,3})^{q}f(\lambda-1,\mu-1,\omega)
\\
&\quad+q(a_1a_{2,3}-a_2a_{1,3})^{q-1}\Big (((aab)-a_{1,2}b_3)f(\lambda,\mu-1,\omega)
\\
&\quad+((abb)-(a_3b_{1,2}))f(\lambda-1,\mu,\omega)-
(ab)f(\lambda,\mu,\omega-1) \Big)
\\
&=q(a_1a_{2,3}-a_2a_{1,3})^{q-1}\Big((q+\lambda+\mu+\omega)f(\lambda,\mu,\omega)
\\
&\quad-a_{1,2}b_3f(\lambda,\mu-1,\omega)-b_{1,2}a_3f(\lambda-1,\mu,\omega)
\Big)
\\
&\quad+
(b_2b_{1,3}-b_1b_{2,3})(a_1a_{2,3}-a_2a_{1,3})^{q}f(\lambda-1,\mu-1,\omega).
\end{split}
\end{equation}
Hence in the case  $p\!<\!q$ after the substitution
$.\mid_{a_{1}a_{2,3}=a_{2}a_{1,3},\,\,b_{1}b_{2,3}=b_{2}b_{1,3}}$ one obtains ~$0$.  Since    $O_a^p $  and  $O_b^q$ commute,  then in the case  $p>q$ one also obtains~$0$.

Now  put  $p=q$.  Then instead of  $\lambda$ and  $\mu$ take
$\lambda-q$ and   $\mu-q$ as in   \eqref{oar}.  According to \eqref{obr}
one obtaines
\begin{equation}
\label{levo1}
q!\!\!\!\sum_{q_1\!+q_2\!+q_3\!=q}\!\!\!(-1)^{q_2+q_3}f(\lambda\!-\!q\!-\!q_2,\mu\!-\!q\!-\!q_3,\omega)
(a_3b_{1,2})^{q_2}(b_3a_{1,2})^{q_3} h^{\lambda,\mu,\omega}_{q_1,q_2,q_3}\!\!+\dots,
\end{equation}
where
\begin{align}
\begin{split}
\label{defh} h^{\lambda,\mu,\omega}_{q_1,q_2,q_3}=\sum_{\substack{\{1,\dots,q\}=
		I_1\sqcup I_2\sqcup I_3,\\ |I_j|=q_j}}&\,\prod_{j\notin I_2\sqcup
	I_3}\!\!\Big((q-j)
\\
&\!+\!(\lambda\!-\!\{\text{ the number of $i_s\!\in\! I_2$, such that
	$i_s\!<\!j$}\})
\\
&\!+\!(\mu\!-\!\{\text{the number of  $i_s\!\in\! I_3$, such that
	$i_s\!<\!j$}\})\!+\!\omega\Big).
\end{split}
\end{align}
Here $\dots$ in   \eqref{levo1}   is a sum of terms that vanish after the substitution
$.\mid_{a_{1}a_{2,3}=a_{2}a_{1,3},\,\,b_{1}b_{2,3}=b_{2}b_{1,3}}$.
After this substitution one gets
\begin{align}
\begin{split}
\label{levo2}
\sum_{}
&a_1^{\varphi_1+\omega_1-\psi_1}a_2^{\varphi_2+\omega_2+\psi_1}a_3^{\varphi_3+q_2}\cdot
a_{1,2}^{\psi_3+q_3}a_{1,3}^{\psi_2+\psi_1}
\\&
\times
b_1^{\psi_1+\omega_2-\varphi_1}b_2^{\psi_2+\omega_1+\varphi_1}b_3^{\psi_3+q_3}\cdot
b_{1,2}^{\varphi_3+q_2}b_{1,3}^{\varphi_2+\varphi_1}
\\
&\times \frac{q! \cdot
	(-1)^{q_2+q_3}\cdot h^{\varphi,\psi,\omega}_{q_1,q_2,q_3}\cdot (-1)^{\varphi_2+\psi_2+\omega_2}}{\varphi_1!
	\varphi_2!\varphi_3!\psi_1!\psi_2!\psi_3!\omega_1!\omega_2!}.
\end{split}
\end{align}
Compare  \eqref{pravo} and  \eqref{levo1}, \eqref{levo2}.
We conclude that in    \eqref{rel2} one has  ${p_{\tau}\!=\!q_{\tau}\!=\!q}$,
and for these summands
\begin{align*}
\begin{cases}
u=\varphi_3+q_2,\quad  v=\psi_3+q_3,\quad \
g=\psi_3+q_3,\quad  h=\varphi_3+q_2,
\\
\theta_1=\varphi_1+\omega_1-\psi_1,\quad
\theta_2=\varphi_2+\omega_2+\psi_1,\quad \theta_{1,3}=\psi_2+\psi_1,\quad \theta_{2,3}=0,
\\
\vartheta_1=\psi_1+\omega_2-\varphi_1,\quad \vartheta_2=\psi_2+\omega_1+\varphi_1,\quad \vartheta_{1,3}=\varphi_2+\varphi_1,\quad \vartheta_{2,3}=0,
\\
X_{\tau}=\frac{(-1)^{\varphi_2+\psi_2+\omega_2+q_2+q_3}}{\Pi_{q,\theta^{\tau}}\Pi_{q,\vartheta^{\tau}}}\cdot
\frac{q!}{\varphi_1!\varphi_2!\varphi_3!\psi_1!\psi_2!\psi_3!\omega_1!\omega_2!}\frac{h^{\Phi,\Psi,\omega}_{q_1,q_2,q_3}}{F_{\theta}
	({\bf 1})F_{\vartheta} ({\bf 1})},
\end{cases}
\end{align*}
where $  h^{\lambda,\mu,\omega}_{q_1,q_2,q_3}$ was defined in \eqref{defh}.

The summation is taken over all indices such that
$$
\begin{cases}
q=q_1+q_2+q_3,\\
\varphi_1+\varphi_2+\phi_3=\lambda-q_2-q,
\\
\psi_1+\psi_2+\psi_3=\mu-q_3-q,
\\ \omega_1+\omega_2=\omega.
\end{cases}
$$

\subsubsection{The third main statement}
We need a formula for the action of  $\frac{E_{3,2}^n}{n!}=
\sum_{n_1+n_2=n}a^{n_1}_3 a_{1,3}^{n_2}\frac{\partial^{n_1}}{\partial^{n_1} a_{2}}\frac{\partial^{n_2}}{\partial^{n_2} a_{1,2}}$ onto a function  $\frac{a_3^{m_1-k_1}}{(m_1-k_1)!}\frac{a_{1,2}^{k_2}}{k_2!}F_{\gamma}$,  associated with a Gelfand-Tselin diagram.
\begin{equation}
\label{etd}
\frac{E_{3,2}^n}{n!}\frac{a_3^{m_1-k_1}}{(m_1-k_1)!}\frac{a_{1,2}^{k_2}}{k_2!}F_{\gamma}=\sum_{n=n_1+n_2}\frac{a_3^{m_1-k_1+n_1}}{(m_1-k_1)!n_1!}\frac{a_{1,2}^{k_2-n_2}}{(k_2-n_2)!}
\frac{a_{1,3}^{n_2}}{n_2!}F_{\gamma-n_1e_2}.
\end{equation}
According to   \eqref{l1}, one has
\begin{align}
\begin{split}
\label{etd1}
\sum_{n=n_1+n_2}&\frac{a_3^{m_1-k_1+n_1}}{(m_1-k_1)!n_1!}\frac{a_{1,2}^{k_2-n_2}}{(k_2-n_2)!}\frac{a_{1,3}^{n_2}}{n_2!}F_{\gamma-n_1e_2}
\\&
=\sum_{\tau}
Z_{\tau
}a_3^{i_{\tau}}a_{1,2}^{j_{\tau}}(a_{1}a_{2,3}-a_{2}a_{1,3})^{r_{\tau}}F_{\varepsilon^{\tau}}.
\end{split}
\end{align}

Thus the relation  \eqref{rel3}  takes place.  Let us prove the Proposition  \ref{pz},  that gives formulas for coefficients in \eqref{rel3}.

Let us use the  ruler   \eqref{fz}. Apply the operator  $O^r$  to the left side of  \eqref{etd1}.
Using  $OF_{\gamma}=0$,
one gets
$$
\frac{a_3^{m_1-k_1+n_1}}{(m_1-k_1)!n_1!}
\frac{a_{1,2}^{k_2-n_2}}{(k_2-n_2)!}\frac{a_{1,3}^{n_2-r}}{(n_2-r)!}F_{\gamma-(n_1+r)e_2}.
$$
Apply  $O^r$  to the right side of  \eqref{etd1}, one gets
$$
\sum_{r_{\tau}=r}\Pi_{r.\varepsilon^{\tau}}\cdot Z_{\tau }\cdot a_3^{i_{\tau}}a_{1,2}^{j_{\tau}}F_{\varepsilon^{\tau}}+\dots,
$$
where $ \dots $  iis a sum of terms that vanish after the substitution   $.\mid_{a_{1}a_{2,3}=a_{2}a_{1,3}}$.

For summands with  $r_{\tau}=r$ one has
\begin{align*}
i_{\tau}&=m_1-k_1+n_1,\quad j_{\tau}=k_2-n_2,\\
\varepsilon_1^{\tau}&=\gamma_1,\quad
\varepsilon_2^{\tau}=\gamma_2-n_1-r,\\
\varepsilon_{1,3}^{\tau}&=\gamma_{1,3}+n_2-r,\quad
\varepsilon_{2,3}^{\tau}=0,\\
&Z_{\tau }=\frac{1}{\Pi_{r,\epsilon^{\tau}}}
\frac{F_{\gamma-n_2e_2}({\bf 1})}{F_{\varepsilon^{\alpha}}({\bf
		1})}\cdot \frac{1}{(m_1-k_1)!n_1!(k_2-n_2)!n_2!}.
\end{align*}

\end{document}